%% file: lambda_hermitian.tex
\documentclass[10pt,a4paper]{article}

\include{canevas_anglais}

\newcommand{\GWN}[1][A,\sigma]{\hat{GW}_{\N}^{\bullet}(#1)}
\newcommand{\GWZ}[1][A,\sigma]{\hat{GW}_{\Z}^{\bullet}(#1)}
\newcommand{\GW}[1][A,\sigma]{\tld{GW}^{\bullet}(#1)}
\newcommand{\SWN}[1][A,\sigma]{\hat{SW}_{\N}^{\bullet}(#1)}
\newcommand{\SWZ}[1][A,\sigma]{\hat{SW}_{\Z}^{\bullet}(#1)}
\newcommand{\SW}[1][A,\sigma]{\tld{SW}^{\bullet}(#1)}
\newcommand{\tw}[1]{\mbox{}^\iota #1}
\DeclareMathOperator{\RAlt}{RAlt}
\DeclareMathOperator{\alt}{alt}
\DeclareMathOperator{\Sh}{Sh}

\author{Nicolas Garrel}
\title{Lambda-operations for hermitian forms over algebras with involution}
\date{}

\begin{document}

\maketitle

\section*{Introduction}

The theory of $\lambda$-rings was initiated by Grothendieck and Berthelot (\cite{Ber})
in the early days of $K$-theory, in particular in the context of the Riemann-Roch theorem,
and has grown to be a field of independent interest, as illustrated for instance in 
the monograph \cite{Yau}, though still often connected with $K$-theory (but see for 
instance \cite{Bor} for an interesting take on $\lambda$-rings as related to the
field with one element).

A $\lambda$-ring is a commutative ring $R$ endowed with operations $\lambda^d: R\to R$
for all $d\in \N$, which are usually understood as a certain flavour of "exterior power",
especially when $R$ is some $K$-theory ring, and that should satisfy some conditions which 
encapsulate the expected behaviour of exterior powers. Note that there has been a shift of
terminology over the years, and what was originally called a $\lambda$-ring is now 
usually called a pre-$\lambda$-ring (and the term $\lambda$-ring refers to what was initially
called a \emph{special} $\lambda$-ring). In this article we will restrict to studying a 
pre-$\lambda$-ring structure, and we will leave the $\lambda$-ring property to a later article.

Since $\lambda$-operations tend to be defined on $K$-theory rings, it is not surprising
that a structure of $\lambda$-ring can be defined on the Grothendieck-Witt ring $GW(K)$ 
of a field $K$ (see \cite{McGar}) since this ring is the $0$th hermitian $K$-theory ring
of $(K,\Id)$. The corresponding "exterior power" operations were introduced by Bourbaki in 
\cite{Bou}. Given a bilinear space $(V,b)$, we get a bilinear space $(\Lambda^d(V),\lambda^d(b))$:
\[ \foncdef{\lambda^d(b)}{\Lambda^d(V)\times \Lambda^d(V)}{K}
  {(u_1\wedge\cdots \wedge u_d, v_1\wedge\cdots \wedge v_d)}{\det\left(b(u_i,v_j)\right).} \]

Those operations, though very natural, make suprisingly few appearances in
the quadratic form literature; for instance, they do not even get a passing
mention in references such as \cite{Lam}, \cite{EKM}, \cite{OMea} or \cite{Szy}. This might
be in part due to the fact that they are not well-defined on the Witt ring, which
is traditionally the preferred algebraic structure for working with quadratic forms,
but rather on the Grothendieck-Witt ring $GW(K)$. The fact that the $\lambda$-structure of
$GW(K)$ was only investigated formally as recently as \cite{McGar}, compared to the 
technically much more difficult theorems on topological $K$-theory, is another illustration
of how $\lambda$-powers of bilinear forms have somehow stayed under the radar.

If one wants to extend the $\lambda$-ring structure of $K_0(A)$ to non-commutative
rings, things do not look great: $K_0(A)$ is not even a ring, as there is no
tensor product of $A$-modules. Rather, if we fix some commutative ring $R$ and an 
$R$-algebra $A$, and if $M$ and $N$ are $A$-modules on the right, then $M\otimes_R N$
is a module over $A\otimes_R A$. This means we can at least define an $\N$-graded ring
$\bigoplus_{d\in \N} K_0(A^{\otimes d})$ where the tensor product is over $R$.
But in general this ring has no reason to be commutative, let alone carry $\lambda$-operations.
Of course that can happen sometimes: for instance, if $A$ is Morita-equivalent to $R$,
then replacing $A$ by $R$ does not change the isomorphism class of the ring, and
therefore it has the necessary structure. That condition on $A$ exactly means that 
$A$ is a split Azumaya algebra over $R$. It turns out that the construction still works
for a non-split Azumaya algebra.

We are more interested here in the hermitian case, and also for simplicity we will
work over a field. Everything in this article still holds over a commutative ring
instead of a field, at the cost of enough technical details that we prefer to first
expose the constructions in the simpler setting of a base field. The non-hermitian
theory then becomes rather uninteresting as $K_0(A)\simeq \Z$ for any Azumaya algebra
$A$ over a field, but the hermitian version $GW(A,\sigma)$ for $(A,\sigma)$ an 
Azumaya algebra with involution over a field with involution $(K,\iota)$ is very rich. 
So, given such an algebra with involution, we can consider 
\[ \GWN = \bigoplus_{d\in \N} GW^\bullet(A^{\otimes d},\sigma^{\otimes d}) \]    
which is the hermitian $K$-theoretical version of the ring above. Here 
$GW^\bullet(A,\sigma) = \bigoplus_{\eps\in U(K,\iota)} GW^\eps(A,\sigma)$ where 
$U(K,\iota)=\ens{\eps\in K^\times}{\eps\iota(\eps)=1}$. This $\GWN$ is an 
$\N\times U(K,\iota)$-graded ring for any algebra with involution, and we showed in 
\cite{Moi2} that it is commutative when $A$ is Azumaya. In this article we show that 
it also has a natural pre-$\lambda$-ring structure, which is graded in the sense that 
if $x$ has degree $g$ then $\lambda^d(x)$ has degree $d\cdot g$ (Theorem \ref{thm_lambda_swn}). 
Precisely, if $(V,h)$ is an $\eps$-hermitian space over $(A^{\otimes n},\sigma^{\otimes n})$, 
we define $(\Alt^d(V),\Alt^d(h))$ (Definition \ref{def_alt_h}) as an $\eps^d$-hermitian space over 
$(A^{\otimes dn},\sigma^{\otimes dn})$ (here $\Alt$ stands for "alternating",
as we prefer to refer to those operations as "alternating powers" rather than
"exterior powers" for reasons which only truly matter in characteristic $2$).
This construction is heavily inspired by the construction of $\lambda$-powers
of involutions given in \cite[§10.D]{BOI}. This graded $\lambda$-structure is also 
shown to be compatible with hermitian Morita equivalences.

There is also a $(\Z\times U(K,\iota))$-graded version $\GWZ$, where the negative
degrees correspond to twisting the algebra using the involution. Its ring structure
is certainly interesting, but we will see that as far as the $\lambda$-structure goes
it does not bring anything more that $\GWN$ (though $\GWZ$ becomes crucial if one
wants to consider duality theorems, which we will do in an upcoming article). 
On the other hand, when the involution
is of the first kind, a much more interesting construction appears, namely a ring
$\GW$ which is graded over $\Zd\times \mu_2(K)$. In that case,
\[ \GW = GW(K)\oplus GW^{-1}(K) \oplus GW(A,\sigma) \oplus GW^{-1}(A,\sigma), \]   
and if $x\in GW^\eps(A,\sigma)$, $\lambda^d(x)$ is in $GW(K)$ when $d$ is 
even, and is in $GW^\eps(A,\sigma)$ when $d$ is odd. 

Though $\GW$ is the more interesting pre-$\lambda$-ring in applications, the
construction of alternating powers and the proof of their properties really happens
in $\GWN$, and actually more precisely in the graded semiring $\SWN$, where
the homogeneous components are isometry classes of hermitian spaces (rather than
formal differences of those). This leads us to study pre-$\lambda$-semirings
which are graded over a commutative monoid. The theory of $\lambda$-operations
over either a semiring or a graded ring has, as far as we know, never been 
formally studied, and graded (semi)rings over monoids are somewhat rare in
the litterature (usually an $\N$-graded ring is seen as a special kind of 
$\Z$-graded ring with trivial negative components). Therefore Section 
\ref{sec_graded_lambda} develops the theory of such graded pre-$\lambda$-semirings
in details. 

In Section \ref{sec_gw} we recall (without proof) the definition and relevant
properties of our various graded semi(rings), from $\SWN$ to $\GW$, referring
to \cite{Moi2} (note that there are some differences in notation between
this article and \cite{Moi2}).

Section \ref{sec_alt} is the technical heart of the article, and is dedicated to
the $\lambda$-structure of $\SWN$ (see Theorem \ref{thm_lambda_swn}), which is
then transferred to the other (semi)rings (Corollaries \ref{cor_lambda_Z} and
\ref{cor_lambda_zd}).

In Section \ref{sec_ralt} we give a more explicit description of the 
$\lambda$-operations in $\SW$, which are initially defined in Section \ref{sec_alt}
from $(\Alt^d(V),\Alt^d(h))$ through a natural Morita equivalence. 
The corresponding "reduced" altenating powers are denoted $(\RAlt^d(V),\RAlt^d(h))$,
which are symmetric bilinear spaces when $d$ is even, and hermitian spaces
when $d$ is odd. Of special interest to us is the connexion between even 
$\lambda$-powers and involution trace forms (Corollary \ref{cor_lambda_trace}).

The last short section is a discussion of the notion of determinant of an 
involution, which is defined in \cite{BOI} only when the algebra has even degree,
and for which we propose an extension to the odd degree case.
\\

We also wish to mention in this introduction our main source of motivation
for defining our $\lambda$-operations: the construction of cohomological
invariants of classical groups and algebras with involution (in the sense
of \cite{GMS}).

Indeed, the proof of Milnor's conjecture
by Voevodsky et al gives a canonical morphism $I^n(K)\to H^n(K,\mu_2)$, where 
$I^n(K)$ is the $n$th power of the fundamental ideal $I(K)\subset W(K)$ of the Witt ring
and $H^n(K,\mu_2)$ is Galois cohomology. Thus, to define a degree $n$ cohomological invariant,
it is possible to define instead an invariant with values in $I^n$. 

For instance, Rost in \cite{R99} and Garibaldi in \cite{Gar} define cohomological
invariants of Spin groups, using well-chosen combinations of $\lambda$-operations
of quadratic forms in $I^3$. In \cite{Moi}, we extend this strategy
to describe all cohomological invariants of $I^n$. To define in this manner invariants
of algebras with involutions (or, more or less equivalently, of hermitian
spaces over those algebras), we therefore need to be able to attach 
quadratic forms to those objects in a natural manner. The most common such construction 
is given by trace forms: if $(A,\sigma)$
is an algebra with involution of the first kind, we can define
the trace form $T_A:x\mapsto \Trd_A(x^2)$, the involution trace form
$T_\sigma: x\mapsto \Trd_A(x\sigma(x))$, its restriction $T_\sigma^+$ to
the subspace of $\sigma$-symmetric elements, and its restriction
$T_\sigma^-$ to the subspace of anti-symmetric elements. These forms
are related by $T_A=T_\sigma^+-T_\sigma^-$ and $T_\sigma=T_\sigma^++T_\sigma^-$,
so it is enough to know $T_\sigma^+$ and $T_\sigma^-$. They have indeed
been used to define or compute some cohomological invariants, for
instance in \cite{BMT}, \cite{AQM2} or \cite{RST}.

It turns out that the involution trace form $T_\sigma$ is nothing more
than the square $\fdiag{1}_\sigma^2$ in $\GW$, while $T_\sigma^\pm$ is 
essentially the same as $\lambda^{2}(\fdiag{1}_\sigma)$ (Corollary \ref{cor_lambda_2}).
But our construction of a $\lambda$-structure on $\GW$ gives a lot
of other new quadratic forms, namely the $\lambda^{2d}(\fdiag{1}_\sigma)$
for $d>1$ (when $2d=\deg(A)$ this is actually nothing more than the
determinant of $\sigma$). A possible strategy to define cohomological
invariants of $(A,\sigma)$ is then to consider well-chosen combinations
of those $\lambda^{2d}(\fdiag{1}_\sigma)$, so that they actually take
values in $I^n$ for some $n$. We show in an upcoming article that this
does work when the index of $A$ is $2$.

\section*{Preliminaries and conventions}

\subsubsection*{Commutative monoids}

Let $M$ be a commutative monoid. We write $M^\times$ for the subgroup of invertible
elements (even when $M$ is denoted additively). We say that a submonoid $N\subseteq M$ 
is saturated (also sometimes called "pure" or "unitary" in the literature) if whenever 
$x+y\in N$ with $x\in N$ then $y\in N$; if $M$ is actually
a group, this exactly means that $N$ is a subgroup.

We write $G(M)$ for the Grothendieck group of $M$, which we recall is generated
by formal differences of elements of $M$; there is always a monoid morphism
$M\to G(M)$ but it is only injective if $M$ satisfies the cancellation property.

We use $\mathbf{ComMon}$ for the category of commutative monoids, and 
$\mathbf{AbGp}$ for the category of abelian groups.





\subsubsection*{Field with involution}

We fix throughout the article a base field $k$ of characteristic not $2$, 
and an étale $k$-algebra $K$, endowed with an involutive automorphism $\iota$
with fixed points $k$. So either $k=K$
and $\iota$ is the identity, or $K$ is a quadratic étale $k$-algebra.
All algebras and modules are assumed to be finite-dimensional over $k$. 
Although it is possible that $K\simeq k\times k$, we will usually pretend that 
$K$ is always a field, and speak of $K$-vector spaces and their dimension,
for instance (you may check that all $K$-modules in this article have constant
rank so this does not cause any trouble). You may exclude the case $K\simeq k\times k$
if you do not want to think about this.

We write $U(K,\iota)=\ens{\eps\in K^\times}{\eps\iota(\eps)=1}$ for the group of unitary
elements of $(K,\iota)$. If $\iota=\Id_K$, then $U(K,\Id)=\mu_2(K)$.

\subsubsection*{Azumaya algebras with involution}

We say that $(A,\sigma)$ is an Azumaya algebra with involution over $(K,\iota)$
when $A$ is an Azumaya algebra over $K$, and $\sigma$ is an involution on $A$
whose restriction to $K$ is $\iota$. Note that in the terminology of, for instance,
\cite{First23}, this would be called an Azumaya algebra with involution over $k$, but
here we do want to fix $(K,\iota)$ and not simply $k$. 

Also, we we choose not to speak of "central simple algebras" to take into account 
that $K$ might not be a field, exactly in the case that $K\approx k\times k$. In
that case, if we fix an identification $K\simeq k\times k$, we get a canonical
isomorphism $(A,\sigma)\simeq (E\times E^{op}, \eps)$ where $E$ is
a central simple algebra over $k$ and $\eps$ is the exchange involution (see
\cite[2.14]{BOI}). All in all we are exactly in the setting of \cite{BOI}.

We write $\Trd_A:A\to K$ for the reduced
trace of $A$, and $\Nrd_A:A\to K$ for its reduced norm. Note that 
if $V$ is a right $A$-module, its reduced dimension $\rdim(V)$ is characterized 
by $\deg(A)\rdim(V)=\dim_K(V)$ (when $K\simeq k\times k$, we technically get
a reduced dimension at each of the two points in $\Spec(K)$, but we will work 
with hermitian modules, where those two dimensions coincide); if $V$ is non-zero, 
it is the degree of the Azumaya algebra $\End_A(V)$. 

Recall that $\sigma$ is of the first kind if $\iota=\Id_K$, and of the
second kind (or unitary) if $\iota$ has order $2$, and that involutions of the first
kind can be orthogonal or symplectic. In particular, if $\iota=\Id_K$ then 
$(K,\Id)$ is an algebra with orthogonal involution. 

For $\eps\in U(K,\iota)$, we define the set $\Sym^\eps(A,\sigma)\subset A$ of 
$\eps$-symmetric elements of $(A,\sigma)$, meaning that they satisfy $\sigma(a)=\eps a$, 
and $\Sym^\eps(A^\times,\sigma)$ is the subset of invertible $\eps$-symmetric elements.

\subsubsection*{Hermitian forms}

Let $(A,\sigma)$ be an Azumaya algebra with involution over $(K,\iota)$.  
If $\eps \in U(K,\iota)$, an $\eps$-hermitian
module $(V,h)$ over $(A,\sigma)$ is a \emph{right} $A$-module $V$, together
with an $\eps$-hermitian form $h:V\times V\to A$ (always assumed to be nondegenerate).
We take the convention that $\sigma(h(x,y))=\eps h(y,x)$.
We often just speak of an $\eps$-hermitian form without mentioning the underlying
module. 

If $V$ is non-zero, $h$ induces the so-called adjoint involution $\sigma_h$ on the
central simple algebra $\End_A(V)$, characterized by $h(u(x),y)=h(x,\sigma_h(u)(y))$.
We call $\eps(h)=\eps$ the sign of $h$ (in the unitary case we need not have $\eps=\pm 1$
but the terminology is still convenient).

If $a\in \Sym^\eps(A^\times,\sigma)$, the elementary diagonal $\eps$-hermitian form
$\fdiag{a}_\sigma: A\times A\to A$ is defined by $(x,y)\mapsto \sigma(x)ay$. A
diagonal form $\fdiag{a_1,\dots,a_n}_\sigma$ is then the (orthogonal) sum of the
$\fdiag{a_i}_\sigma$. When $(A,\sigma)=(K,\Id)$, we remove the subscript $\sigma$
and just write $\fdiag{a_1,\dots,a_n}$ for diagonal bilinear forms.

We define the monoid $SW^\eps(A,\sigma)$ of isometry classes of $\eps$-hermitian 
forms over $(A,\sigma)$ (the zero module is included). We then define $GW^\eps(A,\sigma)$
as the Grothendieck group of $SW^\eps(A,\sigma)$. We often omit the superscript when
$\eps=1$. Note that using tensor products over $K$, $SW(K,\iota)$ is a commutative 
semiring, and $GW(K,\iota)$ is a commutative ring.

\section{Graded pre-$\lambda$-semirings}\label{sec_graded_lambda}

The goal of this article is to define and study an appropriate structure of graded
pre-$\lambda$-ring on the various flavours of mixed Grothendieck-Witt rings of an
Azumaya algebra with involution: $\GWZ$, and $\GW$ if $\sigma$ is of the first
kind (see Section \ref{sec_gw} for the definition of those rings). But ultimately 
this comes from a similar structure on 
$\SWN$, which is only a semiring, graded over a monoid. We wish to insist on the fact
that this is the suitable framework to build the theory, and showcase there is no need 
for a grading over a group (this would change if we wished to include duality results).

We do not assume that the reader is familiar with $\lambda$-rings or with rings
graded over monoids, and we try to give a self-contained account of what is needed
for the article. We take \cite{Yau} as our main reference for the classical theory of 
(ungraded) $\lambda$-rings (though we also sometimes refer to \cite{Zib}). We make all 
the necessary adjustments to take the gradings into account (working with semirings
instead of rings poses no problem whatsoever), and refer directly to the proofs
in \cite{Yau} when they are completely straightforward to adapt to our context.

\subsection{Graded semirings}

If $M$ is a commutative monoid (which we usually denote additively), an $M$-graded commutative 
monoid $A$ is a commutative monoid endowed with a decomposition $A = \bigoplus_{g\in M} A_g$. 
Any ungraded commutative monoid $A$ can be seen as a trivially $M$-graded monoid, by
setting $A_0=A$ and $A_g=0$ if $g\neq 0$. In particular, ungraded monoids are essentially 
the same thing as monoids graded over the trivial monoid.
The elements of each $A_g$ are called homogeneous, and the set of homogeneous elements
is denoted $|A|$. The degree map $\partial: |A|\to M\cup \{\infty\}$ sends $a\in A_g\setminus \{0\}$
to its degree $g\in M$, and $\partial(0)=\infty$ (where $\infty$ is a formal element).
A subset of $A$ is said to be homogeneous if it contains the homogeneous components 
(ie the component in each $A_g$) of all its elements. 

If $A$ and $B$ are $M$-graded, then a graded morphism $f: A\to B$ is a 
monoid morphism such that $f(A_g)\subseteq B_g$ for all $g\in M$. Given a 
function $\phi: M\to N$, if $A$ is $M$-graded and $B$ is $N$-graded, 
we define the pushforward $\phi_*(A)$ as the $N$-graded monoid given
by $\phi_*(A)_h = \bigoplus_{\phi(g)=h}A_g$ for $h\in N$, and the pullback 
$\phi^*(B)$ as the $M$-graded monoid given by $\phi^*(B)_g=B_{\phi(g)}$
for $g\in M$. Note that as ungraded monoids, $\phi_*(A)=A$, so in
particular if $\phi: M\to \{0\}$ is the trivial morphism, then $\phi_*(A)$ is just
$A$ seen as a trivially graded ring. A lax graded morphism $A\to B$
is the data of some function $\phi: M\to N$, and a graded morphism $f: \phi_*(A)\to B$,
which is the same as a graded morphism $A\to \phi^*(B)$.
We also say that $f$ is a $\phi$-graded morphism. Let us write $\mathbf{ComMon}_M$
for the category of $M$-graded commutative monoids with graded morphisms.

Recall that a semiring is the same as a ring except that its underlying additive
structure is only that of a commutative monoid, not necessarily a group.
An $M$-graded semiring is a semiring $R$ which is $M$-graded as an additive monoid,
such that  $1\in R_0$ (the neutral component), and $R_g\cdot R_h\subseteq R_{g+h}$
for any $g,h\in M$. All graded semirings in this article will be commutative.
A (lax) graded semiring morphism is a (lax) graded morphism which is also
a semiring morphism (for a lax morphism of semirings, we require that the 
function $\phi:M\to N$ be a monoid morphism). Note that any ungraded semiring is 
naturally a graded semiring for the trivial grading. We write $\mathbf{SRing}_M$
(resp. $\mathbf{Ring}_M$) for the category of $M$-graded semirings (resp. rings).

The subset $|R|\subset R$ is actually a multiplicative submonoid, and
$\partial: |R|\to M \cup \{\infty\}$ is a monoid morphism (where $m+\infty =\infty$ for
all $m\in M$). An element $x\in |R|$ is called graded-invertible
if for any $g\in M$, multiplication by $x$ induces an additive isomorphism from
$R_g$ to $R_{g+\partial(x)}$. The set of graded-invertible elements is denoted by
$R^\times$ (which agrees with the usual notation if $R$ is ungraded), and it is
a saturated submonoid of $|R|$. A homogeneous element $x\in |R|$ is invertible
if and only if it is graded-invertible and $\partial(x)$ is invertible in $M$;
in particular, if $M$ is a group, then $R^\times = |R|^\times$ (the group
of invertible elements of the monoid $|R|$). On the other hand, in general
an invertible element of $R$ need not be homogeneous, and therefore an invertible
element is not always graded-invertible.

If $R$ is an $M$-graded commutative semiring and $N$ is a commutative
monoid, then the monoid semiring $R[N]$ is a commutative 
$(M\times N)$-graded semiring. In particular, if $R$ is ungraded, $R[M]$ is $M$-graded.
Note that $|R[N]|\simeq |R|\times N$ as monoids, and 
$(R[N])^\times \simeq R^\times \times N$. If $S$ is a commutative 
$(M\times N)$-graded $R$-semialgebra, and $S^\times \xrightarrow{\partial} M\times N\to N$
is surjective, then any set-theoretic section $N\to S^\times$ defines an 
isomorphism of additive $(M\times N)$-graded monoids $S\approx R[N]$, but
that only defines an isomorphism of graded semirings if we can find a section
as a monoid morphism.

An augmentation on a commutative $M$-graded semiring $R$ is a graded morphism
$\tld{\delta}_R: R\to \Z[M]$, and a morphism of augmented graded semirings is a 
graded morphism which preserves the augmentation. We also define the total
augmentation $\delta_R:R\to \Z$, which is an ungraded semiring morphism, as the composition 
of $\tld{\delta}_R$ and the canonical ring morphism $\Z[M]\to \Z$ sending every
element of $M$ to $1$.

\begin{ex}\label{ex_sw_graded}
  Our recurring example in this section will be 
  \[ SW^\bullet(K,\iota) = \bigoplus_{\eps\in U(K,\iota)} SW^\eps(K,\iota) \]
  which is obviously an $U(K,\iota)$-graded additive monoid. It is actually
  an $U(K,\iota)$-graded semiring, using the tensor product over $K$ as the
  multiplication. Indeed, if $(V_i,h_i)$ are $\eps_i$-hermitian modules over
  $(K,\iota)$ for $i\in \{1,2\}$, then $(V_1\otimes_K V_2, h_1\otimes h_2)$
  is an $\eps_1\eps_2$-hermitian module over $(K,\iota)$.

  If $\iota= \Id_K$, $SW^\bullet(K,\Id)$ is $\mu_2(K)$-graded, with two very
  different components: $SW(K,\Id)$ is the semiring of symmetric bilinear forms
  over $K$, while $SW^{-1}(K,\Id)$ is rather uninteresting as antisymmetric forms
  are classified by their (even) dimension. An element of $SW(K,\Id)$ has the form
  $\fdiag{a_1,\dots,a_n}$ with $a_i\in K^*$, and the product is determined by 
  $\fdiag{a}\cdot \fdiag{b}=\fdiag{ab}$. An element of $SW^{-1}(K,\Id)$ has the
  form $r\cdot \mathcal{H}_{-1}$ where $r\in \N$ and $\mathcal{H}_{-1}$ is the
  antisymmetric hyperbolic plane; we have $\fdiag{a}\cdot \mathcal{H}_{-1}=\mathcal{H}_{-1}$
  and $\mathcal{H}_{-1}\cdot \mathcal{H}_{-1}=2\mathcal{H}_{1}$ where $\mathcal{H}_1$
  is the \emph{symmetric} hyperbolic plane. The graded-invertible elements are the
  elementary forms $\fdiag{a}\in SW(K,\Id)$.

  If $\iota\neq \Id_K$, then any element of $SW^\eps(K,\iota)$ has the form 
  $\fdiag{a_1,\dots,a_n}_\iota$ with $a_i\in K^\times$ such that $\iota(a)=\eps a$.
  We can then write any element of $SW^\bullet(K,\iota)$ as $\fdiag{a_1,\dots,a_n}_\iota$
  with $a_i\in K^\times$, with the understanding that each elementary form 
  $\fdiag{a_i}_\iota$ is in the component of degree $\eps_i = \frac{\iota(a_i)}{a_i}
  \in U(K,\iota)$. The $\fdiag{a_i}_\iota$ are precisely the graded-invertible elements,
  with $\fdiag{a}_\iota\cdot \fdiag{b}_\iota = \fdiag{ab}_\iota$,
  and the Hilbert 90 theorem guarantees that there is a graded-invertible element
  $\fdiag{a_\eps}_\iota$ for each
  $\eps\in U(K,\iota)$, so each $SW^\eps(K,\iota)$ is isomorphic as an additive
  monoid to $SW(K,\iota)$, but non-canonically (the choice of $a_\eps$ is modulo $k^\times$,
  but $\fdiag{a}_\iota=\fdiag{b}_\iota$ only if $a\equiv b$ modulo $N_{K/k}(K^\times)$). And 
  in fact the morphism $\partial: (SW^\bullet(K,\iota))^\times\to U(K,\iota)$ can be rewritten
  as $K^\times/N_{K/k}(K^\times)\to K^\times/k^\times$, which does not split in general, 
  so $SW^\bullet(K,\iota)$ and $SW(K,\iota)[U(K,\iota)]$ are not isomorphic as graded
  semirings (not even non-canonically). Despite this, it is true that $SW^\bullet(K,\iota)$ does not carry 
  much more information than $SW(K,\iota)$, which explains why it is rarely defined,
  but it is convenient for us to have a unified treatment of the unitary and non-unitary
  case.

  In the split case $K\simeq k\times k$, the situation is simpler and $SW^\bullet(K,\iota)$
  is canonically isomorphic to $\N[k^\times]$, with $SW(K,\iota)\simeq \N$ and $U(K,\iota)\simeq k^\times$.
\end{ex}

\subsection{Pre-$\lambda$-semirings}

\begin{defi}\label{def_semi_lambda}
  Let $M$ be a commutative monoid. An $M$-graded pre-$\lambda$-semiring is an $M$-graded
  commutative semiring $R$ endowed with functions $\lambda^d:R\to R$ for all $d\in \N$
  such that: 
  \begin{itemize}
    \item for all $g\in M$ and $d\in \N$, $\lambda^d(R_g)\subseteq R_{dg}$;
    \item for all $x\in R$, $\lambda^0(x)=1$ and $\lambda^1(x)=x$;
    \item for all $g\in M$, $x,y\in R_g$ and $d\in \N$, $\lambda^d(x+y)=\displaystyle\sum_{p+q=d}\lambda^p(x)\lambda^q(y)$.
    \end{itemize}
\end{defi}

\begin{rem}\label{rem_def_semi_lambda}
  It is easy to see that one may simply define functions $\lambda^d:R_g\to R_{dg}$
  for all $d\in \N$ and $g\in M$ satisfying the axioms on the homogeneous components, 
  as they extend uniquely to functions $\lambda^d:R\to R$ satisfying the definition.
  The image of a general element is computed from the images of its homogeneous components
  using the axiom for sums.
\end{rem}

\begin{ex}\label{ex_sw_lambda}
  If $(V,h)$ is an $\eps$-hermitian space over $(K,\iota)$, then there is a natural
  $\eps^d$-hermitian form $\lambda^d(h)$ on $\Lambda^d(V)$, given by
  \[ \lambda^d(b)(u_1\wedge\cdots\wedge u_d, v_1\wedge\cdots\wedge v_d) = \det(h(u_i,v_j)). \]
  This defines an $U(K,\iota)$-graded pre-$\lambda$-semiring structure on $SW^\bullet(K,\iota)$.
\end{ex}

Of course, a (lax) morphism of graded pre-$\lambda$-semirings, which we call a graded $\lambda$-morphism,
is a (lax) graded semiring morphism which commutes with the operations $\lambda^d$. This defines
a category $\lambda-\mathbf{SRing}_M$ of $M$-graded pre-$\lambda$-semirings, 
and also of lax graded pre-$\lambda$-semirings.
\\

We now give a statement which justifies the idea that all the definitions 
introduced so far really happen at the level of an $\N$-grading. 

\begin{prop}\label{prop_check_n}
  Let $R$ be a commutative $M$-graded semiring, and let $\lambda^d:R\to R$ be 
  functions for all $d\in \N$. If for any monoid morphism $\phi: \N\to M$ the $\lambda^d$
  induce an $\N$-graded pre-$\lambda$-ring structure on $\phi^*(R)$, then $R$ is a graded
  pre-$\lambda$-ring. 

  Let $R$ and $S$ be $M$-graded pre-$\lambda$-semirings and $f: R\to S$ be a semiring
  morphism. If for any morphism $\phi: \N\to M$ the function $f$ induces an $\N$-graded 
  $\lambda$-morphism $\phi^*(R)\to \phi^*(S)$, then $f$ is a $\lambda$-morphism.
\end{prop}

\begin{proof}
  For any $g\in M$, we write $\phi_g:\N\to M$ the unique morphism with $\phi_g(1)=g$.
  The fact that $\lambda^d(R_g)\subset R_{dg}$ can be checked in $\phi_g^*(R)$. The
  fact that $\lambda^0=1$ and $\lambda^1=\Id$ can be checked on each $R_g$, and therefore
  on each $\phi_g^*(R)$. Given $x,y\in R_g$, the formula for $\lambda^d(x+y)$
  can be checked in $\phi_g^*(R)$.

  Likewise, the fact that $f$ is a $\lambda$-morphism can be checked on each $R_g\to S_g$,
  and therefore for each $\phi_g^*(R)\to \phi_g^*(S)$.
\end{proof}

\begin{defi}
  Let $R$ be an $M$-graded pre-$\lambda$-semiring, and let $x\in R$. We define the
  $\lambda$-dimension $\dim_\lambda(x)$ of $x$ as the supremum in $\N\cup\{\infty\}$ 
  of all $n\in \N$ such that $\lambda^n(x)\neq 0$. The subset of elements with
  finite $\lambda$-dimension is denoted $R^{f.d.}$.
\end{defi}

We usually just say "dimension" for the $\lambda$-dimension when there is no risk
of confusion. Note that only $0\in R$ has dimension $0$, that 
$\dim_\lambda(x+y)\ppq \dim_\lambda(x)+\dim_\lambda(y)$, and that if $f$ is
a graded $\lambda$-morphism, $\dim_\lambda(f(x))\ppq \dim_\lambda(x)$.
\\

It can be useful to rephrase the definition of a graded pre-$\lambda$-semiring in
a more abstract and compact way. For any commutative $M$-graded semiring $R$, consider the
commutative multiplicative monoid
\begin{equation}
  \Lambda(R) = 1 + tR[[t]] \subset R[[t]]
\end{equation}
where $R[[t]]$ is of course the semiring of formal series over $R$, and define for
any $g\in M$ the submonoid
\begin{equation}
  \Lambda(R)_g = \left\{ \sum a_dt^d\in \Lambda(R) \, | \,  \forall n\in \N,\, a_d\in R_{dg}\right\}.
\end{equation}
Then we get an $M$-graded monoid
\begin{equation}
  \Lambda_M(R) = \bigoplus_{g\in M} \Lambda(R)_g.
\end{equation}

There is a natural graded monoid morphism $\eta_R: \Lambda_M(R)\to R$ (where $R$ is seen
as an additive monoid) which sends a series $\sum a_dt^d$ to $a_1$. Defining functions
$\lambda^d:R_g\to R_{dg}$ for all $g\in M$ and $d\in \N^*$ is the same as defining
a single homogeneous function $\lambda_t: R\to \Lambda_M(R)$, using 
$\lambda_t(x)=1+\sum_{d>0}\lambda^d(x)t^d$, and from the definition
of the monoid structure on $\Lambda(R)$ one can easily check that the $\lambda^d$
define a graded pre-$\lambda$-semiring structure if and only if $\lambda_t$ is an
additive morphism which is a section of $\eta_R$.

If $f: S\to R$ is any $M$-graded semiring morphism, then it induces a commutative
diagram
\[ \begin{tikzcd}
    \Lambda_M(S) \rar{\eta_S} \dar{f_*} & S \dar{f} \\
    \Lambda_M(R) \rar{\eta_R} & R
  \end{tikzcd} \]
and when $R$ and $S$ are graded pre-$\lambda$-semirings, then $f$ is
a $\lambda$-morphism if and only if the following natural diagram
commutes:
\[ \begin{tikzcd}
    S \rar{\lambda_t} \dar{f} & \Lambda_M(S) \dar{f_*} \\
    R \rar{\lambda_t} & \Lambda_M(R).
    \end{tikzcd} \]

An element $x\in R$ is finite-dimensional exactly when $\lambda_t(x)\in R[[t]]$ is
a polynomial, and its dimension is then the degree of this polynomial.

\begin{rem}
  If $\phi: M\to N$ is a monoid morphism, there is a canonical morphism
  $\Lambda_M(R)\to \Lambda_N(\phi_*(R))$, which is compatible with the construction
  of $\eta_R$. This means that a structure of $M$-graded pre-$\lambda$-semiring
  on $R$ canonically induces a structure of $N$-graded pre-$\lambda$-semiring on
  $\phi_*(R)$. In particular, taking $\phi$ to be the trivial morphism $\phi: M\to \{0\}$,
  it induces a structure of (ungraded) pre-$\lambda$-semiring on $R$ (as $\Lambda_{\{0\}}(\phi_*(R))$
  is just $\Lambda(R)$).

\end{rem}

\subsection{Augmentation}

The following proposition is immediate from the pre-$\lambda$-semiring axioms:

\begin{prop}\label{prop_lambda_graded_group_ring}
  If $R$ is an $M$-graded pre-$\lambda$-semiring and $N$ is a commutative monoid, then
  $R[N]$ has a canonical $(M\times N)$-graded pre-$\lambda$-semiring
  structure given by $\lambda^d(x\cdot h)=\lambda^d(x)\cdot (dh)$ for all
  $d\in \N$, $x\in R$ and $h\in N$.

  Moreover, the canonical semiring morphism $R[N]\to R$ is a $\lambda$-morphism.
\end{prop}

\begin{ex}
  There is a canonical pre-$\lambda$-ring structure on $\Z$, given by
  $\lambda^d(n)=\binom{n}{d}$, which then induces a canonical $M$-graded
  pre-$\lambda$-ring structure on the monoid ring $\Z[M]$.
\end{ex}

\begin{defi}
  An augmentation of an $M$-graded pre-$\lambda$-semiring $R$ is an augmentation
  $\tld{\delta_R}:R\to \Z[M]$ which is a $\lambda$-morphism. The total
  augmentation map $\delta_R:R\to \Z$ is then an ungraded $\lambda$-morphism.
\end{defi}

We get a category $\lambda-\mathbf{SRing}_M^+$ (resp. $\lambda-\mathbf{Ring}_M^+$)
of augmented $M$-graded pre-$\lambda$-semirings (resp. rings).

\begin{ex}
  The graded dimension map $SW^\bullet(K,\iota)\to \Z[U(K,\iota)]$, which sends the 
  isometry class of $(V,h)$ in $SW^\eps(K,\iota)$ to $\dim(V)\cdot \eps$, is an 
  augmentation on the graded pre-$\lambda$-semiring $SW^\bullet(K,\iota)$.
\end{ex}

In general, we want to give the augmentation map the interpretation of a "graded
dimension", but it should not be confused with the $\lambda$-dimension.
Since $\binom{n}{d}\neq 0$ for all $d\in \N$ when $n<0$,
an element $x\in R^{f.d.}$ must satisfy $\delta_R(x)\pgq 0$. Also, since the 
$\lambda$-dimension of $n\in \N$ is just $n$, and a $\lambda$-morphism
lowers the $\lambda$-dimension, we see that we must have $0\ppq \delta_R(x)\ppq \dim_\lambda(x)$.

\subsection{Positive structure}

In practice, a lot of (pre)-$\lambda$-rings, such as the $K_0$ ring of a commutative 
ring or a topological space, are defined as Grothendieck rings of semirings (of modules
or vector bundles in those examples), which means that general elements are formal 
differences of "concrete" elements which enjoy a better behaviour.

First we take a look at Grothendieck groups and Grothendieck rings. Our first observation
is that the Grothendieck ring of a semiring is "the same thing as" the Grothendieck group
of its underlying additive monoid. More precisely, if $R$ is a commutative semiring, 
and $G(R)$ is its additive Grothendieck group, there is a unique ring structure on $G(R)$
such that the canonical map $R\to G(R)$ is a semiring morphism, and this uniquely defined
ring $G(R)$ satisfies the expected universal property that any (commutative) semiring
morphism $R\to S$ where $S$ is actually a ring extends uniquely to a ring morphism 
$G(R)\to S$, and on the level of additive monoids, this is the map given by the universal
property of the Grothendieck group.

Let us consider the following commutative diagram of categories and functors
(which as always in those situations only commutes up to a natural isomorphism):
\begin{equation}\label{eq_diagram_groth}
  \begin{tikzcd}
    \mathbf{SRing} \rar & \mathbf{ComMon} \\
    \mathbf{Ring} \rar \arrow[hook]{u} & \mathbf{AbGp} \arrow[hook]{u}
  \end{tikzcd}
\end{equation}
The vertical maps are inclusion of subcategories, and the horizontal
ones are forgetful functors to the additive structure. The vertical inclusions
are actually reflexive, with reflectors the Grothendieck group/ring construction.
Then what we discussed above can be formalized in this more general setting: let
\[ \begin{tikzcd}
  \mathbf{C} \rar{U} & \mathbf{D} \\
  \mathbf{A} \rar{V} \arrow[hook]{u}{I} & \mathbf{B} \arrow[hook]{u}{J}
\end{tikzcd} \]
be a commutative square of functors. We say that it has the reflector extension property
if $U$ and $V$ are faithful, $I$ and $J$ are inclusion of reflexive subcategories,
with respective reflectors $F: \mathbf{C}\to \mathbf{A}$ and $G: \mathbf{D}\to \mathbf{B}$,
and $V\circ F = G\circ U$.

In that situation, given $X\in \mathbf{C}$, let $Y\in \mathbf{A}$ be such that 
$V(Y)=G(U(X))$. We have a canonical arrow $U(X)\to J(G(U(X)))$ given by the adjunction
$G\dashv J$, which defines in turn $U(X)\to J(V(Y))=U(I(Y))$. If this morphism comes
from a (necessarily unique) morphism $X\to I(Y)$, then the associated morphism 
$F(X)\to Y$ is an isomorphism. Applying this to diagram (\ref{eq_diagram_groth})
is exactly what we explained above regarding Grothendieck groups/rings.

It is equally easy to see that if $M$ is a commutative monoid,
the inner squares (and thus the outer square too) of
\[ \begin{tikzcd}
  \mathbf{SRing}_M \rar & \mathbf{ComMon}_M \rar & \mathbf{ComMon} \\
  \mathbf{Ring}_M \rar \arrow[hook]{u} 
     & \mathbf{AbGp}_M \rar \arrow[hook]{u} & \mathbf{AbGp} \arrow[hook]{u}
\end{tikzcd} \]
have the reflector extension property. This means that if $R$ is an $M$-graded commutative semiring,
its $M$-graded Grothendieck ring is just $G(R)=\bigoplus_{g\in M} G(R_g)$ with
a unique compatible structure.

\begin{prop}\label{prop_lambda_groth}
  The squares in the following diagram have the reflector extension property:
  \[ \begin{tikzcd}
    \lambda\text{-}\mathbf{SRing}_M^+ \rar & \lambda\text{-}\mathbf{SRing}_M \rar 
      & \mathbf{SRing}_M \\
      \lambda\text{-}\mathbf{Ring}_M^+ \rar \arrow[hook]{u} 
       & \lambda\text{-}\mathbf{Ring}_M \rar \arrow[hook]{u} & \mathbf{Ring}_M \arrow[hook]{u}.
  \end{tikzcd} \]
\end{prop}

\begin{proof}
  Let $S$ be an $M$-graded pre-$\lambda$-semiring, and let $G(S)$ be its $M$-graded
  Grothendieck ring. We need to show that $G(S)$ has a unique structure of 
  $M$-graded pre-$\lambda$-ring such that the canonical morphism $S\to G(S)$ is
  a $\lambda$-morphism, and that this pre-$\lambda$-ring satisfies the universal property.

  Interpreting the $\lambda$-structure as a monoid morphism $\lambda_t: S\to \Lambda_M(S)$,
  and observing that $\Lambda_M(G(S))$ is actually a group since $G(S)$ is a ring, the
  universal property of Grothendieck groups tells us that there is a unique
  $\lambda_t: G(S)\to \Lambda_M(G(S))$ such that the natural diagram
  \[ \begin{tikzcd}
    S \rar{\lambda_t} \dar & \Lambda_M(S) \dar  \\
    G(S) \rar{\lambda_t} & \Lambda_M(G(S))
  \end{tikzcd} \]
  commutes.

  If $R$ is an $M$-graded pre-$\lambda$-ring and $S\to R$ is a $\lambda$-morphism,
  then we need to check whether the diagram of abelian groups
  \[ \begin{tikzcd}
    G(S) \rar{\lambda_t} \dar & \Lambda_M(G(S)) \dar  \\
    R \rar{\lambda_t} & \Lambda_M(R)
  \end{tikzcd} \]
  commutes. But since both compositions $G(S)\to \Lambda_M(R)$ extend
  $S\to \Lambda_M(R)$, this is true by universal property.

  The case where $S$ is augmented is proved similarly, as the augmentation is
  just the data of a $\lambda$-morphism $S\to \Z[M]$.
\end{proof}

\begin{ex}\label{ex_gw_lambda}
  We define $GW^\bullet(K,\iota)= G(SW^\bullet(K,\iota))$ as augmented 
  $U(K,\iota)$-graded pre-$\lambda$-rings. The component of degree $\eps\in U(K,\iota)$
  is simply $GW^\eps(K,\iota)=G(SW^\eps(K,\iota))$.
\end{ex}

As we mentioned earlier, when $R$ is a pre-$\lambda$-ring generated additively
by some semiring $S\subseteq R$, we would like $S$ to enjoy good properties, 
that will somewhat extend to $R$. We adapt the treatment in \cite{Zib} to formalize 
those properties:

\begin{defi}\label{def_rigid}
  Let $S$ be an augmented $M$-graded pre-$\lambda$-semiring. We say that $S$
  is \emph{rigid} if it is cancellative as an additive monoid, and:
  \begin{enumerate}
  \item if $x\in |S|$ satisfies $\delta_S(x)=0$, then $x=0$;
  \item every $1$-dimensional homogeneous element is graded-invertible.
  \end{enumerate}
  If $S$ is rigid, we write $\ell(S)$ for the set of $1$-dimensional 
  homogeneous elements, and we call those \emph{line elements}.
\end{defi}

\begin{ex}
  Our usual example $SW^\bullet(K,\iota)$ is rigid, and actually we even get
  $\ell(SW^\bullet(K,\iota))=(SW^\bullet(K,\iota))^\times$. Indeed, we have
  stated in Example \ref{ex_sw_graded} that the graded-invertible elements
  are the $1$-dimensional forms, which are the line elements. The other condition
  is clear: a $0$-dimensional form is just the zero element.
\end{ex}

\begin{lem}\label{lem_dim_rigid}
  Let $S$ be a rigid augmented $M$-graded pre-$\lambda$-semiring. Then
  for any $x\in |S|$, we have $\delta_S(x)=\dim_\lambda(x)$.
\end{lem}

\begin{proof}
  Let $d=\delta_S(x)$. We already know that $d\ppq \dim_\lambda(x)$. 
  Now since $\delta_S$ is a $\lambda$-morphism, $\delta_S(\lambda^{d+1}(x))
  = \binom{d}{d+1}=0$ so $\lambda^{d+1}(x)=0$, and $d= \dim_\lambda(x)$.
\end{proof}

\begin{defi}
  Let $R$ be an augmented $M$-graded pre-$\lambda$-semiring. A \emph{positive 
  structure} on $R$ is the data of a sub-structure $R_{\pgq 0}\subseteq R$ such
  that:
  \begin{itemize}
    \item $R_{\pgq 0}$ is rigid;
    \item $(R_{\pgq 0})^\times \subset R^\times$;
    \item if $x\in |R|$, there are $a,b\in |R_{\pgq 0}|$ such that $x+a=b$.
  \end{itemize}
  The elements of $R_{\pgq 0}$ are called \emph{positive}, and we also set 
  $R_{> 0}=R_{\pgq 0}\setminus \{0\}$. We also write $\ell(R)=\ell(R_{\pgq 0})$,
  and still call those elements the line elements of $R$.

  To keep terminology short, we say that an augmented $M$-graded pre-$\lambda$-semiring
  with positive structure is an $M$-structured semiring.
\end{defi}

Note that there is an obvious category of $M$-structured semirings, which
preserve the positive structure, and also a similar category with lax
$\lambda$-morphisms.

\begin{ex}
  Clearly, any rigid augmented $M$-graded pre-$\lambda$-semiring is $M$-structured,
  with itself as the positive structure.
\end{ex}

\begin{ex}
  If $R$ is an $M$-structured semiring, then $R[N]$ is an $(M\times N)$-structured semiring,
  with $(R[N])_{\pgq 0} = (R_{\pgq 0})[N]$ and $\ell(R[N])=\ell(R)\times N$.

  As $\N\subset \Z$ is a positive structure for $\Z$, $\N[M]$ is a positive structure
  for $\Z[M]$ with $\ell(\Z[M])=M$.
\end{ex}

\begin{prop}\label{prop_pos_groth}
  Let $S$ be a rigid $M$-structured semiring. Then it is a positive structure
  on $G(S)$.
\end{prop}

\begin{proof}
  Note that since $S$ is cancellative, the canonical $S\to G(S)$ is injective,
  and we can treat it as an inclusion. The only thing to show is that 
  $S^\times\subset G(S)^\times$. If $a\in S^\times$, then multiplication 
  by $a$ induces isomorphisms $S_g\to S_{g+\partial(a)}$ for all $g\in M$,
  and therefore by functoriality induces isomorphisms 
  $G(S_g) \to G(S_{g+\partial(a)})$.
\end{proof}

\begin{ex}
  This means that $SW^\bullet(K,\iota)$ is a positive structure on $GW^\bullet(K,\iota)$,
  and this is the structure we have in mind when we say that $GW^\bullet(K,\iota)$
  is an $U(K,\iota)$-structured ring.
\end{ex}

\begin{lem}\label{lem_line_satur}
  Let $R$ be an $M$-structured semiring. Then $\ell(R)$ is a saturated
  submonoid of $R^\times$, and therefore a saturated submonoid of the
  multiplicative monoid $|R|$.
\end{lem}

\begin{proof}
  We have by definition that $\ell(R)\subseteq (R_{\pgq 0})^\times
  \subseteq R^\times$. Lemma \ref{lem_dim_rigid} shows that $\ell(R)$ is a 
  submonoid, since $\delta_R$ is multiplicative. Since $R^\times$ is saturated 
  in $|R|$, it is enough to show that $\ell(R)$ is saturated in $R^\times$.
  
  Let $x,y,z\in R^\times$ such that $xy=z$ and $x,z\in \ell(R)$.
  We first show that $y$ is positive. Since $x$ is graded-invertible 
  in $R_{\pgq 0}$ and $\partial(z)=\partial(x)+\partial(y)$,
  we may write $z=x\cdot y'$ with $y'$ positive of degree $\partial(y)$, and since
  $x$ is graded-invertible in $R$, then $y=y'$.

  Then the equality $\delta_R(x)\delta_R(y)=\delta_R(z)$
  gives $1\times \delta_R(y)=1$ so $\dim_\lambda(y)=1$ by Lemma
  \ref{lem_dim_rigid}, and by definition $y$ is a line element.
\end{proof}

Note that positive elements have finite dimension since they are sums of homogeneous 
positive elements. The positive structure ensures that $R$ enjoys a well-behaved theory 
of dimension:

\begin{prop}\label{prop_pos_dim}
  Let $R$ be an $M$-structured semiring. Then for any element $x\in R^{f.d.}$, 
  we have $\dim_\lambda(x)=\delta_R(x)$, and the leading coefficient of $\lambda_t(x)$ is a 
  line element. In particular, all elements of $\lambda$-dimension $1$ are line elements,
  and $\dim_\lambda$ is an additive function on $R^{f.d.}$.
\end{prop}

\begin{proof}
  Let $A\subseteq R$ be the subset of elements $x$ such that $\lambda_t(x)$ is a polynomial
  of degree $\delta_R(x)$ whose leading coefficient is a line element. By definition
  $A\subseteq R^{f.d}$, and we want to show that they are actually equal, which takes care
  of all statements in the proposition.

  First, we see that $|R_{\pgq 0}|\subseteq A$. Indeed, if $x$ is a positive
  homogeneous element, then $\dim_\lambda(x)=\delta_R(x)$ by Lemma \ref{lem_dim_rigid},
  and if this dimension is $n$, then $\lambda^n(x)$ is a line element because it
  is positive and has dimension $\binom{n}{n}=1$.

  Then, we see that $A$ is stable by sum: if $x,y\in A$, then $\lambda_t(x+y)
  =\lambda_t(x)\lambda_t(y)$, so if $at^n$ and $bt^m$ are the leading terms of 
  $\lambda_t(x)$ and $\lambda_t(y)$ respectively, the leading term of
  $\lambda_t(x+y)$ is $abt^{n+m}$ with $ab\in \ell(R)$ according to Lemma
  \ref{lem_line_satur} (it is the leading term since $ab\neq 0$, which can
  be deduced from $ab\in \ell(R)$). Note that $n+m$ is 
  $\delta_R(x)+\delta_R(y) =\delta_R(x+y)$.

  This shows that $R_{\pgq 0}\subseteq A$. Now let $z\in R^{f.d}$, and let us 
  write $x+z=y$ with $x, y\in R_{\pgq 0}$ (in particular, $x,y\in A$).
  Let $at^n$, $bt^m$ and $ct^r$ be the leading terms of $\lambda_t(x)$,
  $\lambda_t(y)$, and $\lambda_t(z)$ respectively. Since $a\in \ell(R)\subseteq R^\times$,
  we have $ac\neq 0$, so $ac$ is the leading coefficient of $\lambda_t(x)\lambda_t(z)$,
  and thus $ac=b$ since $\lambda_t(y)=\lambda_t(x)\lambda_t(z)$. Since $\ell(R)$ is 
  saturated in $|R|$ by Lemma \ref{lem_line_satur}, we get $c\in \ell(R)$. Also
  the degree of $\lambda_t(z)$ is $r=n-m$ which is $\delta_R(x)-\delta_R(y)=\delta_R(z)$,
  and we do get $z\in A$.
\end{proof}

\begin{rem}\label{rem_struct_dim}
  A morphism of $M$-structured rings preserves the augmentation, therefore 
  it preserves the dimension of finite-dimensional elements, and also induces
  a monoid morphism between line elements.
\end{rem}

\subsection{Determinant}\label{sec_lambda_det}

We saw in proposition \ref{prop_pos_dim} that if $R$ is $M$-structured and $x\in R$
has dimension $n$, then $\lambda^n(x)$ is a line element. This construction can be
extended to all elements, not just finite-dimensional ones, but the price to pay is that
we need to consider $G(\ell(R))$ instead of $\ell(R)$.

\begin{rem}
  Recall that if $M$ is actually a group, then $R^\times$ is a group,
  as well as $\ell(R)$ (as it is a saturated submonoid), so in that case
  $G(\ell(R))=\ell(R)$.
\end{rem}

\begin{prop}\label{prop_def_det}
  Let $R$ be an $M$-structured semiring. There is a unique monoid morphism
  \[ \det: R\to G(\ell(R)), \]
  which we call the \emph{determinant}, such that if $\dim_\lambda(x)=n\in \N$, 
  then $\det(x)=\lambda^n(x)$. If $f$ is a morphism of $M$-structured rings, 
  then $\det(f(x))=f(\det(x))$.
\end{prop}

\begin{proof}
  We have from Proposition \ref{prop_pos_dim} a well-defined function 
  $\det: R^{f.d}\to \ell(R)$. It is a monoid morphism since $\det(x)$ is the
  leading coefficient of $\lambda_t(x)$ and $\lambda_t(x+y)=\lambda_t(x)\lambda_t(y)$
  (and since the leading coefficients are line elements, their product is non-zero).

  This then extends uniquely to $G(R^{f.d})\to G(\ell(R))$. But now note that
  \[ R_{\pgq 0}\subseteq R^{fd}\subseteq R\subseteq G(R_{\pgq 0})=G(R^{f.d})=G(R) \]
  so we have our unique extension to the additive monoid $R$.

  The compatibility with morphisms is easy to see since they preserve the
  $\lambda$-dimension (see Remark \ref{rem_struct_dim}) and $f(\lambda^n(x))=\lambda^n(f(x))$
  if $x$ is of finite dimension $n$.
\end{proof}

\begin{ex}\label{ex_gw_det}
  We have seen that $\ell(GW^\bullet(K,\iota))\simeq K^\times/N_{K/k}(K^\times)$
  (when $K=k$, this has to be understood as $K^\times/(K^\times)^2$). For an
  $\eps$-hermitian space $(V,h)$, the above notion of determinant then coincides
  with the classical one: the Gram determinant of an orthogonal basis is determined
  by the isometry class of $(V,h)$ only up to an element of $N_{K/k}(K^\times)$,
  and the corresponding class in $K^\times/N_{K/k}(K^\times)$ is the determinant
  of $(V,h)$.

  When $\iota=\Id_K$, this is the usual determinant of a bilinear form as a square
  class, and note that for an anti-symmetric form the determinant is always
  trivial. When $\iota\neq \Id_K$, the determinant is usually only given a 
  definition for $\eps=1$, where it then takes values in $k^\times/N_{K/k}(K^\times)$.
\end{ex}

\subsection{Contractions}

In this section we explain how to transfer structure and properties between
the various flavours of mixed Grothendieck-Witt (semi)rings. The basic connexion
between these different versions is:

\begin{defi}
  Let $R$ be an $M$-graded semiring and let $\phi: M\to N$ be a surjective
  monoid morphism. A contraction of $R$ along $\phi$ is a $\phi$-graded
  morphism $R\to S$ such that for all $g\in M$, $R_g\to S_{\phi(g)}$ is
  an isomorphism.
\end{defi}

\begin{prop}\label{prop_contr_lambda}
  Let $f: R\to S$ be a contraction along some $\phi: M\to N$. It defines a 
  relation between graded pre-$\lambda$-semiring structures on $R$ and on $S$, 
  such that two such structures are in relation if $f$ is a lax $\lambda$-morphism.
  Then this relation is actually a bijection between graded 
  pre-$\lambda$-semiring structures on $R$ and $S$.

  Likewise, the contraction defines a bijection between augmented structures
  on $R$ and $S$, and under this correspondence $R$ is rigid if and only
  $S$ is. In particular, this also gives a bijective correspondence between
  structures of $M$-structured semiring on $R$, and $N$-structured semiring
  on $S$.
\end{prop}

\begin{proof}
  For any $g\in M$ and $d\in \N$, $f$ induces isomorphisms $R_g\isom S_g$ and
  $R_{dg}\to S_{dg}$, so clearly a system of functions $\lambda^d: R_g\to R_{dg}$
  uniquely determines a system $\lambda^d: S_g\to S_{dg}$ and conversely. Also,
  from the axioms of $\lambda$-operations it is clear that one is a 
  pre-$\lambda$-structure if and only if the other one is too.

  Likewise, under the isomorphisms $R_g\to S_g$, functions $R_g\to \Z$ and
  $S_g\to \Z$ are in bijective correspondence, and one is an augmentation if
  and only if the other one is.

  It is easy to see that for any $x\in |R|$, $x$ is quasi-invertible (resp. a 
  line element) if and only if $f(x)$ is, which shows that $R$ is rigid if and
  only $S$ is.
\end{proof}

\section{Mixed Grothendieck-Witt rings}\label{sec_gw}

In this section, we review the definitions and results from \cite{Moi2} about the
mixed Grothendieck-Witt ring which are necessary for our purposes. We adopt slightly 
different conventions, which we will explain, but it is completely straightforward to
adapt the results, so we just refer to \cite{Moi2} for all results in this section.

\begin{defi}
  Let $(A,\sigma)$ and $(B,\tau)$ be Azumaya algebras with involution over $(K,\iota)$.
  A hermitian Morita equivalence from $(B,\tau)$ to $(A,\sigma)$ is a $B$-$A$-bimodule $V$
  endowed with a regular $\eps$-hermitian form $h:V\times V\to A$ over $(A,\sigma)$,
  with $\eps\in U(K,\iota)$, 
  such that the action of $B$ on $V$ induces a $K$-algebra isomorphism
  $B\simeq \End_A(V)$, under which $\tau$ is sent to the adjoint involution $\sigma_h$
  (which means that $h(bu,v)=h(u,\tau(b)v)$).
\end{defi}

There exists such an equivalence if and only if $A$ and $B$ are Brauer-equivalent;
in this case, the isomorphism class of the bimodule $V$ is unique, and if we fix
such a $V$, the $\eps$-hermitian form $h$ is unique up to a multiplicative scalar: if $h'$
is another choice, there is some $\lambda\in K^\times$ such that $h'=\fdiag{\lambda}h$.

\begin{defi}\label{def_br_h}
  The hermitian Brauer 2-group $\CBrhu$ of $(K,\iota)$ is the category whose objects
  are Azumaya algebras with involutions over $(K,\iota)$, and morphisms $(B,\tau)\to (A,\sigma)$
  are isomorphism classes of $\eps$-hermitian Morita equivalences from $(B,\tau)$ to
  $(A,\sigma)$.

  The composition of $(U,g):(C,\theta)\to (B,\tau)$ and $(V,h): (B,\tau)\to (A,\sigma)$
  is defined as $(U\otimes_B V, f)$ with
  \[ f(u\otimes v, u'\otimes v') = h(v, g(u,u')v'). \]
  If $g$ is $\eps_1$-hermitian and $h$ is $\eps_2$-hermitian, then $f$ is 
  $\eps_1\eps_2$-hermitian.
\end{defi}

Note that the identity of $(A,\sigma)$ in $\CBrhu$ is the diagonal form $(A,\fdiag{1}_\sigma)$.
It can be shown that all morphisms are invertible. Specifically, if $(V,h)$ is a morphism
from $(B,\tau)$ to $(A,\sigma)$, then we can define an $A$-$B$-bimodule $\bar{V}$ as
being $V$ as a $K$-vector space, but with twisted action $a\cdot v\cdot b=
\tau(b)\cdot v\cdot \sigma(a)$. Then we have a natural $\eps(h)$-hermitian form 
$\bar{h}$ on $\bar{V}$ over $(B,\tau)$ defined by $\bar{h}(x,y)z=xh(y,z)$ for all 
$x,y,z\in V$, and the inverse of $(V,h)$ in $\CBrh$ is $(\bar{V},\fdiag{\eps(h)}\bar{h})$,
which is $\eps(h)^{-1}$-hermitian.

The association $(A,\sigma)\mapsto SW^\bullet(A,\sigma)$ defines a functor from
$\CBrhu$ to the category of $U(K,\iota)$-graded commutative monoids with lax morphisms.
Precisely, if $\eps\in U(K,\iota)$ and $(V,h)$ is an equivalence from $(B,\tau)$ to
$(A,\sigma)$, composition with $(V,h)$ induces an isomorphism $SW^\eps(B,\tau)\to
SW^{\eps\eps(h)}(A,\sigma)$, and therefore a $(\eps\mapsto \eps\eps(h))$-graded
isomorphism $SW^\bullet(B,\tau)\isom SW^\bullet(A,\sigma)$.

Actually, both $\CBrhu$ and the category of $U(K,\iota)$-graded commutative monoids
are naturally symmetric monoidal categories, and $SW^\bullet$ is a symmetric monoidal
functor. This is simply encoded by a natural map
\[ SW^\eps(A,\sigma)\otimes SW^{\eps'}(B,\tau)\to 
SW^{\eps\eps'}(A\otimes_K B,\sigma\otimes\tau) \]
which is just given by the tensor product of hermitian modules.

The general machinery of \cite{Moi2} then provides for each Azumaya algebra
with involution $(A,\sigma)$ a commutative $\Gamma_\N$-graded semiring, where 
$\Gamma_\N= \N\times U(K,\iota)$:
\begin{equation}
  \SWN = \bigoplus_{(d,\eps)\in \Gamma_\N} SW^\eps(A^{\otimes d},\sigma^{\otimes d}) 
 = \bigoplus_{d\in \N} SW^\bullet(A^{\otimes d},\sigma^{\otimes d})
\end{equation}
where by convention $(A^{\otimes 0},\sigma^{\otimes 0})=(K,\iota)$.
This actually defines a functor from $\CBrhu$ to $\Gamma_\N$-graded semirings
with lax morphisms: for any $\eps\in U(K,\iota)$, let $\phi_\eps: \Gamma_\N\to \Gamma_\N$
be the monoid morphism $\phi_\eps(d,\eps')=(d,\eps^d\eps')$; then an equivalence
$(V,h)$ from $(B,\tau)$ to $(A,\sigma)$ induces a $\phi_{\eps(h)}$-graded isomorphism
$\SWN[B,\tau]\to \SWN$, which is the direct sum of the isomorphisms 
$SW^\bullet(B^{\otimes d},\tau^{\otimes d})\isom SW^\bullet(A^{\otimes d},\sigma^{\otimes d})$
induced by $(V^{\otimes d},h^{\otimes d})$, for all $d\in \N$.
\\

In $\CBrhu$, each $(A,\sigma)$ has a "weak inverse", given by the conjugate algebra
$(\mbox{}^\iota A,\mbox{}^\iota \sigma)$. Here $\mbox{}^\iota A$ is $A$ as a ring, but
with the $K$-algebra structure given by $K\xrightarrow{\iota} K\to A$ (twisting the 
$K$-algebra structure by $\iota$), and $\mbox{}^\iota \sigma$ is just $\sigma$ as a 
function (the notation is just here to keep track of twisting). Precisely, there is a 
canonical Morita equivalence
$(A\otimes_K \mbox{}^\iota A, \sigma\otimes \mbox{}^\iota \sigma)\to (K,\iota)$, given
by $(|A|_\sigma,T_\sigma)$, where $|A|_\sigma$ is the left $(A\otimes_K \mbox{}^\iota A)$-module 
which is $A$ as a vector space, with "twisted" action
\begin{equation}\label{eq_twisted_action}
  (a\otimes b)\cdot x = ax\sigma(b)
\end{equation}
and $T_\sigma$ is the involution trace form
\begin{equation}\label{eq_inv_trace}
  T_\sigma(x,y) = \Trd_A(\sigma(x)y).
\end{equation}

Again, the machinery of \cite{Moi2} then provides a commutative $\Gamma_\Z$-graded semiring,
where $\Gamma_\Z=\Z\times U(K,\iota)$:
\begin{equation}
  \SWZ = \bigoplus_{(d,\eps)\in \Gamma_\Z} SW^\eps(A^{\otimes d},\sigma^{\otimes d}) 
 = \bigoplus_{d\in \Z} SW^\bullet(A^{\otimes d},\sigma^{\otimes d}).
\end{equation}
where by convention $(A^{\otimes d},\sigma^{\otimes d})=
(\mbox{}^\iota A^{\otimes d},\mbox{}^\iota \sigma^{\otimes d})$ when $d<0$.
Again this is functorial in $(A,\sigma)$, as an $\eps$-hermitian equivalence induces
a $\phi_\eps$-graded isomorphism of semirings, where $\phi_\eps: \Gamma_Z\to \Gamma_\Z$
extends the previous version on $\Gamma_\N$: $\phi_\eps(d,\eps')=\eps^d\eps'$, where this
time $d\in \Z$.

This can be seen as a gluing of $\SWN$ and $\SWN[\mbox{}^\iota A,\mbox{}^\iota \sigma]$
identifying the two copies of $SW^\bullet(K,\iota)$ in each semiring. The non-trivial
ingredient is that one can multiply forms of positive and negative $\Z$-degree, such that
those degrees cancel each other out. The most important example is that of degrees $1$ and
$-1$, where the morphism
\[  SW^\eps(A,\sigma)\times SW^{\eps'}(\mbox{}^\iota A,\mbox{}^\iota \sigma)\to 
SW^{\eps\eps'}(K,\iota) \]
is induced by the equivalence 
$(A\otimes_K \mbox{}^\iota A, \sigma\otimes \mbox{}^\iota \sigma)\to (K,\iota)$
explained above.

\begin{ex}
  Let $\fdiag{a}_\sigma\in SW^\eps(A,\sigma)$ and $\fdiag{b}_{\mbox{}^\iota \sigma}\in 
  SW^{\eps'}(\mbox{}^\iota A,\mbox{}^\iota \sigma)$. Note that this means that, due to
  the twisting of $(\mbox{}^\iota A,\mbox{}^\iota \sigma)$, we have 
  $\sigma(b)=\iota(\eps')b$. Then
  \[ \fdiag{a}_\sigma \cdot \fdiag{b}_{\mbox{}^\iota \sigma} = 
  T_{\sigma,a,b}\in SW^{\eps\eps'}(K,\iota) \]
  where $T_{\sigma,a,b}:A\times A\to K$ is the $\eps\eps'$-hermitian form defined
  as
  \[ T_{\sigma,a,b}(x,y) = \Trd_A(\sigma(x)ay\sigma(b)). \]
  In particular, $\fdiag{1}_\sigma \cdot \fdiag{1}_{\mbox{}^\iota \sigma} = T_{\sigma}$.
\end{ex}

\begin{rem}
  If $f: (B,\tau)\to (A,\sigma)$ is a morphism in $\CBrhu$, the induced morphism
  $f_*: \SWZ[B,\tau]\to \SWZ$ is simply the identity of $SW^\bullet(K,\iota)$
  on the $\{0\}\times U(K,\iota)$-components.
\end{rem}

\begin{rem}
  The functoriality implies that the graded semiring $\SWZ$ only
  depends on the Brauer class of $A$, but \emph{noncanonically}: if $A$ and
  $B$ are Brauer-equivalent, then there exists a Morita equivalence between
  $(A,\sigma)$ and $(B,\tau)$ inducing an isomorphism on the graded Grothendieck-Witt
  semirings, but there are several choices of such equivalences, which amount
  to a choice of scaling.
\end{rem}

When $\iota=\Id_K$, something special happens, since $(\mbox{}^\iota A,\mbox{}^\iota \sigma)$
is nothing but $(A,\sigma)$. In that case, reflecting the fact that the Brauer class
 $[A]\in \Br(K)$ has order $2$, we get a canonical isomorphism 
 $(A,\sigma)^{\otimes 2}\to (K,\Id_K)$ in $\CBrhu$. Again, the machinery in \cite{Moi2} then defines
 a commutative $\Gamma$-graded semiring, where $\Gamma = \Zd \times \mu_2(K)$:
 \begin{equation}
  \SW = SW^\bullet(K,\Id) \oplus SW^\bullet(A,\sigma).
\end{equation}
We call this the mixed Grothendieck-Witt semiring. This is again functorial in $(A,\sigma)$,
with induced morphisms being the identity on the component $SW^\bullet(K,\Id)$.

Those three flavours of graded Grothendieck-Witt semirings $\SW$, $\SWN$ and $\SWZ$
are naturally related: we have an obvious commutative triangle of commutative monoids
\[  \begin{tikzcd}
  \Gamma_\N \arrow[hook]{r} \arrow[twoheadrightarrow]{dr} & \Gamma_\Z \arrow[twoheadrightarrow]{d} \\
   & \Gamma
\end{tikzcd} \]
along which we get a lax morphism $\SWN\hookrightarrow \SWZ$ and, when $\iota = \Id_K$,
a natural triangle
\[  \begin{tikzcd}
  \SWN \arrow[hook]{r} \arrow[twoheadrightarrow]{dr} & \SWZ \arrow[twoheadrightarrow]{d} \\
   & \SW
\end{tikzcd} \]
where the morphisms to $\SW$ are contractions. The natural contraction $\SWZ\to \SW$ can
be characterized by the fact that it identifies the two copies of $SW^\bullet(A,\sigma)$
in $\SWZ$ (those in degree $1$ and $-1$).

\begin{ex}\label{ex_square_diag}
  By definition, in $\SW$ we have $\fdiag{1}_\sigma^2 = T_\sigma\in SW(K)$.
\end{ex}

Sending a (hermitian) module to its reduced dimension defines a monoid morphism from
$SW_\eps(A,\sigma)$ to $\N$. They can be bundled together to 
define a graded semiring morphism $\hat{\rdim}: \SWZ \to \N[\Gamma_\Z]$ and a 
"total reduced dimension" morphism $\rdim: \SWZ\to \N$. When $\iota=\Id_K$,
we also get $\tld{\rdim}: \SW \to \N[\Gamma]$ and $\rdim: \SW\to \N$.
\\

By taking Grothendieck rings, we also obtain graded ring versions of our semirings,
namely $\GWN$ and $\GWZ$, and $\GW$ when $\iota=\Id_K$, which are functorial in
$(A,\sigma)$, and satisfy an obvious natural commutative triangle, and the morphisms 
$\GWN\to \GW$ and $\GWZ\to \GW$ are contractions when they make sense.

\begin{prop}\label{prop_gw_mixte_split}
  If $(A,\sigma)=(K,\iota)$, then we have canonical isomorphisms of graded
  (semi)rings $\SWZ[K,\iota]\simeq SW^\bullet(K,\iota)[\Z]$, 
  $\GWZ[K,\iota]\simeq GW^\bullet(K,\iota)[\Z]$.
  When $\iota=\Id_K$ we also get $\SW[K,\Id]\simeq SW^\bullet(K,\Id)[\Zd]$ and 
  $\GW[K,\Id]\simeq GW^\bullet(K,\Id)[\Zd]$.
\end{prop}

\begin{proof}
  This just follows from the elementary observation that $(K,\iota)^{\otimes d}\simeq (K,\iota)$
  for any $d\in \N$, and also when $d<0$ since $\iota$ is an isomorphism $(\tw{K},\tw{\iota})\isom (K,\iota)$.
  Checking that the correpsonding monoid isomorphism
  \[ \bigoplus_{d\in \Z} SW^\bullet(K^{\otimes d},\iota^{\otimes d}) \Isom 
  \bigoplus_{d\in \Z} SW^\bullet(K,\iota) \]  
  is an isomorphism of graded semirings $\SWZ[K,\iota]\isom SW^\bullet(K,\iota)[\Z]$
  is then a simple check. The reasoning is the same for the other isomorphisms.
\end{proof}

\section{$\lambda$-operations on hermitian forms}\label{sec_alt}

The goal of this section is to endow our various semirings with
appropriate $\lambda$-structures. Our work in Section \ref{sec_graded_lambda} will
allow us to restrict our attention to $\SWN$, and then transfer the
structure to the other semirings.

\subsection{Alternating powers of a module}\label{sec_alt_mod}

Let $A$ be an Azumaya algebra over $K$.
The first step is to associate to each $A$-module $V$ an $A^{\otimes d}$-module
$\Alt^d(V)$, such that we recover the construction of the exterior power (or rather
the altenating power, which makes no difference in practice)
in the split case. The natural context of the alternating power construction for
vector spaces is that of Schur functors, but the development of such a theory
for modules over central simple algebras is beyond the scope of this article,
and will be addressed in future work. 

It is still useful to view the exterior power construction as a consequence of
the structure of module over a symmetric group. Namely, if $V$ is a $K$-vector
space, then $V^{\otimes d}$ is naturally a left $K[\mathfrak{S}_d]$-module, and 
$\Lambda^d(V)$ is the quotient of $V^{\otimes d}$ by the subspace generated by the 
kernels of $1-\tau$ for all transpositions $\tau\in \mathfrak{S}_d$. Now if $V$ is
a right $A$-module, it is in particular a $K$-vector space, so $V^{\otimes d}$
still has the left $K[\mathfrak{S}_d]$-module structure given by the permutation
of the $d$ factors, but it is \emph{not} the one we want to use, since it is not
compatible with the action of $A^{\otimes d}$ on the right (to see how ill-suited
this action would be, consider that if $V=A$, the $A^{\otimes d}$-module generated 
by the kernel of any $1-\tau$ is the full $V^{\otimes d}$, since it contains 
$1_A\otimes\cdots\otimes 1_A$).

Instead, recall from \cite[3.5]{BOI} that for any Azumaya $K$-algebra $B$, 
the Goldman element $g_B\in (B\otimes_K B)^\times$ is defined as the pre-image
of the reduced trace map $\Trd_B:B\to K\subseteq B$ under the canonical isomorphism
of \emph{vector spaces}
\[ B\otimes_K B\Isom B\otimes_K B^{op}\Isom \End_K(B), \]
and from \cite[10.1]{BOI} that sending a transposition $(i,\, i+1)\in \mathfrak{S}_d$ 
to $1\otimes\cdots\otimes g_B\otimes\cdots\otimes 1$ extends to a group morphism
$\mathfrak{S}_d\to (B^{\otimes d})^\times$, and thus to a $K$-algebra morphism
\[ K[\mathfrak{S}_d]\to B^{\otimes d}. \]

Now let again $V$ be a (non-zero) right $A$-module, and let $B=\End_A(V)$. Then 
from the canonical algebra morphisms from $K[\mathfrak{S}_d]$ to $B^{\otimes d}$ and
$A^{\otimes d}$, we have a canonical structure of left $K[\mathfrak{S}_d]$-module
on $V^{\otimes d}$ which commutes with the action of $A^{\otimes d}$, and a
canonical structure of right $K[\mathfrak{S}_d]$-module which commutes with 
the action of $B^{\otimes d}$ (in particular, those two actions commute with
one another). Those two actions are by default the ones we have in mind when
we work with $V^{\otimes d}$; they may be called the Goldman action, as
opposed to the permutation action, if it is necessary to make the distinction clear. 
When $V=0$, $\End_A(V)$ is not Azumaya, 
but we of course still have (trivial) actions of $K[\mathfrak{S}_d]$.
Note that both actions are compatible with scalar extension, and the one on
the left is compatible with Morita equivalence. The connection between the
Goldman and permutation action is given by:

\begin{prop}\label{prop_goldman_commut}
  Let $A$ be an Azumaya algebra over $K$, and let $V$ be a right
  $A$-module. Then for any $v_1,\dots,v_d\in V$
  and any $\pi\in \mathfrak{S}_d$:
  \[ \pi(v_1\otimes\cdots \otimes v_d)\pi^{-1} =
  v_{\pi^{-1}(1)}\otimes \cdots \otimes v_{\pi^{-1}(d)}. \]
  If $A=K$, the action of $K[\mathfrak{S}_d]$ on $V^{\otimes d}$ on the right
  is trivial, and its action on the left is the usual permutation action
  on the $d$ factors.
\end{prop}

\begin{proof}
  Let us set $B=\End_A(V)$.
  By construction of the $K[\mathfrak{S}_d]$-module structures, we can reduce
  to the case where $d=2$ and $\pi$ is the transposition, and extending
  the scalars if necessary we may assume that $A$ and $B$ are split.
  
  In this case we have $A\simeq \End_K(U)$, $B\simeq \End_K(W)$,
  and $V\simeq \Hom_K(U,W)$ with obvious actions from $A$ and $B$.
  It is shown in \cite{BOI} that the Goldman element $g_A$ in 
  $A\otimes_K A\simeq \End_K(U\otimes_K U)$ is the switching map (and
  of course likewise for $g_B$). Therefore, if $f_1,f_2\in V$ and $u_1,u_2\in U$:
  \begin{align*}
    (g_B\cdot f_1\otimes f_2)(u_1\otimes u_2) &= g_B(f_1(u_1)\otimes f_2(u_2)) \\
                                       &= f_2(u_2)\otimes f_1(u_1) \\
                                       &= (f_2\otimes f_1)(g_A(u_1\otimes u_2)) \\
                                       &= (f_2\otimes f_1 \cdot g_A)(u_1\otimes u_2)
  \end{align*}
  so indeed $g_B\cdot f_1\otimes f_2 = f_2\otimes f_1\cdot g_A$.

  The last statement is a direct consequence, taking into account that the
  Goldman element of $K$ is $1\in K\otimes K = K$.
\end{proof}

For any finite set $X$, if $\mathfrak{S}_X$ is its symmetric
group and $Y\subset \mathfrak{S}_X$, we define the alternating element
\begin{equation}\label{eq_def_alt}
  \alt(Y) = \sum_{g\in Y} (-1)^gg \in K[\mathfrak{S}_X],
\end{equation}
where $(-1)^g$ is the sign of the permutation $g$.
When $Y=\mathfrak{S}_X$ we call 
\begin{equation}\label{eq_def_sd}
  s_X = \alt(\mathfrak{S}_X) =  \sum_{g \in \mathfrak{S}_X} (-1)^g g,
 \end{equation}
the anti-symmetrizer element of $X$.
In particular, this defines $s_d\in K[\mathfrak{S}_d]$.

\begin{lem}\label{lem_ker_im_sd}
  Let $V$ be a right $A$-module, where $A$ is an Azumaya $K$-algebra.
  Let us identify elements of $K[\mathfrak{S}_d]$ with the maps they induce
  on $V^{\otimes d}$ through the left Goldman action. Then 
  \[ \ker(s_d) = \sum_{g\in \mathfrak{S}_d} \ker(1+(-1)^gg) \]
  and
  \[ \Ima(s_d) = \bigcap_{g\in \mathfrak{S}_d} \ker(1-(-1)^gg). \]
  Moreover, the equalities still hold if we restrict $g$ to a 
  generating set of $\mathfrak{S}_d$.
\end{lem}

\begin{proof}
  In both cases, the equality can be checked after extending the scalars
  to a splitting field, and then by Morita equivalence it can be reduced
  to $A=K$, that is to the case of vector spaces, where it amounts to simple
  combinatorics on a basis, which we spell out explicitly.

  Choose a basis $(e_i)_{1\ppq i\ppq r}$ of $V$. For any $\bar{x}\in \{1,\dots,r\}^d$, 
  we write $e_{\bar{x}}$ for the corresponding basis element of $V^{\otimes d}$,
  and for any $I\subseteq \{1,\dots,r\}$ of size $d$, we define $e_I$ as $e_{\bar{I}}$ 
  where $\bar{I}$ consists of the elements of $I$ in increasing order.
  Then each $\bar{x}$ either has (at least) two equal components, or has 
  the form $g\bar{I}$ for a unique $I$ and a unique $g\in \mathfrak{S}_d$ 
  with support in $I$. In the first case, $s_d e_{\bar{x}}=0$, and in the
  second case $s_d e_{\bar{x}}=(-1)^g s_d e_I$. 
  
  Therefore it is easy to see that the kernel of $s_d$ is generated by 
  the $e_{\bar{x}}$
  where $\bar{x}$ has at least two equal components (in which case it is in
  the kernel of $1-g$ for some transposition $g$), and by the 
  $e_{\bar{x}}-(-1)^ge_{\bar{x}}$ where $\bar{x}$ has distinct components, 
  which is in the image of $1-(-1)^gg$, so in the kernel of $1+(-1)^gg$. 
  This shows that $\ker(s_d) \subseteq \sum_g \ker(1+(-1)^gg)$.
  The reverse inclusion can be seen from $s_d(1+(-1)^gg) = 2s_d$, so 
  since the characteristic of $K$ is not $2$, $\ker(1+(-1)^gg)\subseteq 
  \ker(s_d)$.

  Note that if $g=g_1\cdots g_r$ where the $g_i$ are in some 
  generating set $S$, then 
  \[ e_{\bar{x}}-(-1)^ge_{g\bar{x}} = (e_{\bar{x}} + e_{g_r\bar{x}}) 
  - (e_{g_r\bar{x}} + e_{g_{r-1}g_r\bar{x}}) + \cdots 
  - (-1)^g(e_{g_2\cdots g_r\bar{x}} + e_{g\bar{x}})  \]
  so actually $\ker(s_d) = \sum_{g\in S} \ker(1+(-1)^gg)$.

  It also follows from our earlier computations that the $s_de_I$ form a 
  basis of the image of $s_d$.
  Let $v = \sum a_{\bar{x}}e_{\bar{x}}\in \bigcap_g \ker(1-(-1)^gg)$, and let 
  $\bar{x}\in \{1,\dots,r\}^d$. If $\bar{x}$ has at least two equal
  components then $a_{\bar{x}}=0$ because there is a transposition $g$
  such that $g\bar{x}=\bar{x}$ so $a_{\bar{x}}=-a_{\bar{x}}$ (and we assumed
  that the characteristic of $K$ is not $2$). And if $\bar{x}$ has 
  distinct components, we have $\bar{x}=g\bar{I}$ for some subset $I$,
  and $a_{\bar{x}}=(-1)^ga_{\bar{I}}$. All in all, this means that
  $v = \sum_I a_{\bar{I}}s_de_I$, so $\bigcap_g \ker(1-(-1)^gg)\subseteq \Ima(s_d)$.
  The reverse inclusion is clear since $(1-(-1)^gg)s_d=0$.
  Finally, note that if we have $gv=(-1)^gv$ 
  for all $g$ in some generating set, then it is true for any $g$.
\end{proof}

This lemma shows in particular that, given the classical definition of 
exterior powers, when $V$ is a vector space (so $A=K$) we have a canonical 
identification $s_d V\simeq \Lambda^d(V)$, with 
$s_d(v_1\otimes\cdots\otimes v_d)$ corresponding to 
$v_1\wedge\cdots\wedge v_d$ (but this is only valid in characteristic different 
from $2$, which is why we avoid the term "exterior power" for our construction). 
This motivates the following definition of alternating powers of a module:

\begin{defi}\label{def_alt}
  Let $A$ be an Azumaya algebra over $K$, let $V$
  be a right $A$-module, and let $d\in \N$. We set
  \[ \Alt^d(V) = s_d V^{\otimes d}\subseteq V^{\otimes d} \]
  as a right $A^{\otimes d}$-module, with in particular $\Alt^0(V)=K$
  and $\Alt^1(V)=V$.
\end{defi}

\begin{rem}
  If we wanted to define an analogue of exterior powers that is also valid 
  in characteristic $2$, we could define $\operatorname{Ext}^d(V)$
  as the quotient of $V^{\otimes d}$ by the $A^{\otimes d}$-submodule
  generated by the $\ker(1-(-1)^gg)$ for $g\in \mathfrak{S}_d$. But
  this is not as convenient for our purposes.
\end{rem}

\begin{prop}\label{prop_dim_alt}
  Let $A$ be an Azumaya algebra over $K$, let $V$ be
  a right $A$-module, and let $d\in \N$. Then
  \[ \rdim_{A^{\otimes d}}(\Alt^d(V)) = \binom{\rdim_A(V)}{d}. \]
  In particular, if $d>\rdim_A(V)$ then $\Alt^d(V)$ is the zero module.
\end{prop}

\begin{proof}
  Once again, it is enough to check this when $A$ is split, and then by
  Morita equivalence when $A=K$. But then this is the usual formula for
  the dimension of $\Lambda^d(V)$.
\end{proof}

\begin{rem}
  In \cite[§10.A]{BOI}, the algebra $\lambda^d(A)$ is defined, using our
  notation, as $\End_{A^{\otimes d}}(\Alt^d(A))$, where $A$ is seen as
  a module over itself. When $d\ppq \deg(A)$, $\lambda^d(A)$ is an
  Azumaya algebra, but when $d> \deg(A)$, $\lambda^d(A)$ is the zero ring
  (we will try to avoid using this notation in that case).

  In general, if $B=\End_A(V)$ and $d\ppq \rdim(V)$, then $\End_{A^{\otimes d}}(\Alt^d(V))$
  is canonically isomorphic to $\lambda^d(B)$.
\end{rem}

\subsection{The shuffle product}

In this section, we give an appropriate generalization of the wedge product
on exterior powers of vector spaces, that is to say an associative product
from $\Alt^p(V)\otimes_K \Alt^q(V)$ to $\Alt^{p+q}(V)$.

We start by recalling some elementary results about symmetric groups and 
shuffles. Let us fix a finite totally ordered set $X$, and a partition 
$X = \coprod_i I_i$. Then recall that in $\mathfrak{S}_X$ we have the Young 
subgroup $\mathfrak{S}_{(I_i)}$, consisting of the permutations stabilizing
each $I_i$, and the set of shuffles $Sh_{(I_i)}$, which are the 
permutations whose restriction to each $I_i$ is an increasing function
$I_i\to X$. The usefulness of shuffles is explained by the following lemma:

\begin{lem}\label{lem_shuffle_young}
  Any element of $\mathfrak{S}_X$ can be written in a unique way as $\pi\sigma$, 
  with $\pi\in Sh_{(I_i)}$ and $\sigma\in \mathfrak{S}_{(I_i)}$.
\end{lem}

\begin{proof}
  Let $\tau\in \mathfrak{S}_X$. For any $i$, we can write $I_i = \{a_{i,1},\dots,a_{i,d_i}\}$ 
  such that $a_{i,1}<\dots<a_{i,d_i}$, but also $I_i = \{b_{i,1},\dots,b_{i,d_i}\}$ 
  such that $\tau(b_{i,1})<\dots<\tau(b_{i,d_i})$. Then we define a permutation $\sigma_i$ 
  of $I_i$ by $\sigma_i(b_{i,j})=a_{i,j}$, and since we do it for every $i$ we get 
  a permutation $\sigma\in \mathfrak{S}_{(I_i)}$. 

  Let $\pi=\tau\sigma^{-1}$; by construction $\pi\in Sh_{(I_i)}$ since
  $\pi(a_{i,j}) = \tau(b_{i,j})$ so $\pi(a_{i,1})<\dots <\pi(a_{i,d_i})$. The way
  we defined the $\sigma_i$ and $\pi$ makes it clear that this is the only possible
  decomposition.
\end{proof}

Shuffles also have a nice compatibility with refinements of partitions:

\begin{lem}\label{lem_shuffle_assoc}
  Suppose that each $I_i$ is itself partitioned as 
  $I_i = \coprod_j J_{i,j}$. Let $\tau\in \mathfrak{S}_X$, and let 
  $\tau = \pi\sigma$ be the decomposition given by lemma 
  \ref{lem_shuffle_young}, with $\sigma$ corresponding
  to $(\sigma_i)_i\in \prod_i \mathfrak{S}_{I_i}$. Then $\tau$ is a
  $(J_{i,j})_{i,j}$-shuffle if and only if each $\sigma_i$ is a 
  $(J_{i,j})_j$-shuffle.
\end{lem}

\begin{proof}
  Since $\pi$ is increasing on each $I_i$, it is clear that $\tau$
  is increasing on all $J_{i,j}$ if and only $\sigma_i$ is.
\end{proof}

We can record some very basic observations on the $\alt$ construction and
the shuffle group:

\begin{itemize}
  \item If $A,B,C\subseteq \mathfrak{S}_X$ are such that any element
    of $A$ can be written uniquely as a product of an element of $B$
    and an element of $C$, then $\alt(A)=\alt(B)\alt(C)$.
  \item If $A_i\subseteq \mathfrak{S}_{I_i}$ for each $i$, then $\alt(\prod_i A_i)
    = \bigotimes_i \alt(A_i)$ where we identify $\prod_i \mathfrak{S}_{I_i}
    \simeq \mathfrak{S}_{(I_i)}$ and $K[\mathfrak{S}_{(I_i)}]\simeq 
    \bigotimes_i K[\mathfrak{S}_{I_i}]$.
\end{itemize}

If we now define the shuffle element $sh_{(I_i)}\in K[\mathfrak{S}_X]$ by
\begin{equation}\label{eq_def_shuffle}
  sh_{(I_i)} = \alt(Sh_{(I_i)}) = \sum_{\pi\in Sh_{(I_i)}} (-1)^\pi \pi,
\end{equation}
we get the following consequences:

\begin{coro}\label{cor_shuffle_relat}
  It holds in $K[\mathfrak{S}_X]$ that
  \[ s_X = sh_{(I_i)}\cdot (s_{I_1}\otimes\cdots\otimes s_{I_r}) \]
  and if each $I_i$ is further partitioned as $I_i = \coprod_j J_{i,j}$,
  that
  \[ sh_{(J_{i,j})_{i,j}}=sh_{(I_i)}\cdot 
        (sh_{(J_{1,j})_j}\otimes\cdots\otimes sh_{(J_{r,j})_j}).\]
\end{coro}

\begin{proof}
  Those equalities are corollaries of Lemma \ref{lem_shuffle_young} and
  \ref{lem_shuffle_assoc} respectively, using the two observations above.
\end{proof}

When $X=\{1,\dots,d\}$ and the partition comes from a decomposition
$d= d_1+\dots +d_r$, we simply write $sh_{d_1,\dots,d_r}\in K[\mathfrak{S}_d]$
for the corresponding shuffle element.

This leads to the following definition:

\begin{defi}
  Let $A$ be an Azumaya algebra over $K$, and let $V$ be a right $A$-module. 
  The shuffle algebra of $V$ is defined as the $K$-vector space
  \[ \Sh(V) = \bigoplus_{d\in \N} V^{\otimes d} \]
  (which is the same underlying space as the tensor algebra $T(V)$ over $K$)
  with the product $V^{\otimes p}\otimes_K V^{\otimes q}\to V^{\otimes p+q}$
  defined by
  \begin{equation}\label{eq_prod_shuffle}
    x \# y = sh_{p,q}(x\otimes y),
  \end{equation}
  which we call the shuffle product.
\end{defi}

The term algebra is fully justified by the following proposition:

\begin{prop}\label{prop_shuffle_alg}
  Let $A$ be an Azumaya algebra over $K$, and let $V$ be a right $A$-module.
  The shuffle algebra $\Sh(V)$ is an $\N$-graded associative $K$-algebra with unit 
  $1\in K = V^{\otimes 0}$.

  Furthermore, 
  \[ \Alt(V)=\bigoplus_{d=0}^{\rdim(V)}\Alt^d(V)\subseteq \Sh(V) \] 
  is a subalgebra, and is actually the $K$-subalgebra generated by $V=\Alt^1(V)$. 
  Precisely, if $x\in V^{\otimes p}$ and $y\in V^{\otimes q}$ with $p+q=d$, we have
  \[ (s_p x) \# (s_q y) = s_d(x\otimes y) \]
  and in particular if $x_1,\dots,x_d\in V$:
  \[ x_1\# \cdots \# x_d = s_d(x_1\otimes\dots\otimes x_d). \]
\end{prop}

\begin{proof}
  For the associativity, we make use of corollary \ref{cor_shuffle_relat}:
  if $p+q+r=d$, we have 
  \[sh_{p,q,r} = sh_{p+q,r}\cdot (sh_{p,q}\otimes 1)= sh_{p,q+r}\cdot (1\otimes sh_{q,r})\]
  which by definition of the shuffle product implies that if $x\in V^{\otimes p}$,
  $y\in V^{\otimes q}$ and $z\in V^{\otimes r}$, $(x\# y)\# z = x\#(y\# z)$.
  The claims about the grading and the unit are trivial.

  The rest of the statement follows directly from the formula $(s_p x) \# (s_q y) 
  = s_d(x\otimes y)$, which is a clear consequence of corollary \ref{cor_shuffle_relat}
  since it implies $sh_{p,q}(s_p\otimes s_q)=s_d$.
\end{proof}

From what we already observed previously, when $A=K$ the alternating algebra
$\Alt(V)$ is canonically isomorphic to the exterior algebra $\Lambda(V)$ (again, 
because we avoid characteristic $2$), and
the shuffle product corresponds to the wedge product.

One of the main properties of the wedge product is its anti-commutativity,
and we do get some version of that property:

\begin{prop}\label{prop_shuffle_comm}
  Let $A$ be an Azumaya algebra over $K$, and let $V$ be a right $A$-module.
  For any $d\in \N$, $x_1,\dots,x_d\in V$ and $\pi\in \mathfrak{S}_d$, we have
  \[ x_{\pi(1)}\# \dots \# x_{\pi(d)} = (-1)^\pi (x_1\# \dots \# x_d)\pi. \]
\end{prop}

\begin{proof}
  From proposition \ref{prop_goldman_commut} we see that $x_{\pi(1)}\# \dots \# x_{\pi(d)}$
  is $s_d\pi^{-1}(x_1\otimes\cdots\otimes x_d)\pi$, so we can conclude using
  $s_d g = (-1)^g s_d$ for any $g\in \mathfrak{S}_d$.
\end{proof}

We now establish the analogue of the well-known addition formula for exterior powers 
of vector spaces, which computes the alternating powers of a direct sum $U\oplus V$
in terms of those of $U$ and $V$. 

\begin{prop}\label{prop_alt_sum_subalg}
  Let $A$ be an Azumaya algebra over $K$, and let $U$ and $V$ be right 
  $A$-modules. Then the subspaces $\Sh(U)$ and $\Sh(V)$ of $\Sh(U\oplus V)$
  are subalgebras for the shuffle product, and likewise $\Alt(U)$ and
  $\Alt(V)$ are subalgebras of $\Alt(U\oplus V)$.
\end{prop}

This is an immediate consequence of the following lemma:

\begin{lem}\label{lem_goldman_sum}
  Let $A$ be an Azumaya algebra over $K$, and let $U$ and $V$ be right 
  $A$-modules. Then the restriction of the Goldman action of $\mathfrak{S}_d$
  on $(U\oplus V)^{\otimes d}$ to $U^{\otimes d}$ (resp. $V^{\otimes d}$) is
  precisely the Goldman action on $U^{\otimes d}$ (resp. $V^{\otimes d}$).
\end{lem}

\begin{proof}
  We can reduce by scalar extension to the case where $A$ is split, and by Morita 
  equivalence to $A=K$, in which case the result is about the classical permutation 
  actions, and is therefore clear.
\end{proof}

We now give our result for computing the alternating powers of direct sums:

\begin{prop}\label{prop_add_mod_alt}
  Let $A$ be an Azumaya algebra over $K$, and let $U$ and $V$ be right 
  $A$-modules. Then for any $d\in \N$ the shuffle product induces an isomorphism of
  $A^{\otimes d}$-modules :
  \[ \bigoplus_{p+q=d} \Alt^p(U)\otimes_K \Alt^q(V) \Isom \Alt^d(U\oplus V). \]
\end{prop}

\begin{proof}
  Proposition \ref{prop_alt_sum_subalg} ensures that we can compute all shuffle products
  in $\Alt(U\oplus V)$.

  Using proposition \ref{prop_shuffle_alg}, we easily establish that $\Alt^d(U\oplus V)$ is
  linearly spanned by the elements of the type $x_1 \# \cdots \# x_d$ with $x_i$
  in $U$ or $V$. Now using proposition \ref{prop_shuffle_comm}, we can permute the $x_i$
  so that $x_1,\dots,x_p\in U$ and $x_{p+1},\dots,x_d\in V$, at the cost of multiplying
  on the right by some $\pi\in \mathfrak{S}_d$. But any element of this type
  is obviously in the image of the map described in the statement of the proposition, so
  this map is surjective. We may then conclude that it is an isomorphism by checking the 
  dimensions over $K$ (using proposition \ref{prop_dim_alt} for instance).
\end{proof}

Note that in particular this defines a natural $\N$-graded $K$-linear isomorphism 
between $\Alt(U)\otimes_K \Alt(V)$ and $\Alt(U\otimes V)$, but it is not quite
an algebra isomorphism.

\begin{rem}
  This construction of $\Alt^d(V)$ defines a structure of $\N$-graded
  pre-$\lambda$-ring (and actually $\N$-structured ring) on 
  $\bigoplus_d K_0(A^{\otimes d})$. It is not very impressive 
  when we work over a field since this ring is just $\Z[\N]$,
  but the construction also works over an arbitrary base ring,
  and in that case this is more meaningful.
\end{rem}

\subsection{Alternating powers of a $\eps$-hermitian form}

Now if $V$ is a non-zero $A$-module equipped with a $\eps$-hermitian form
$h$ with respect to some involution $\sigma$ on $A$, we want
to endow $\Alt^d(V)$ with an induced form $\Alt^d(h)$ such that
when $A=K$ we recover the exterior power of the hermitian form.
This requires understanding the interaction between the action of the
symmetric group and the involutions on the algebras. 

Recall that any group algebra $K[G]$ has a canonical involution $S$ given
by $S: g\mapsto g^{-1}$. As $K$ carries the involution $\iota$, we can 
twist $S$ to $S_\iota$ which still acts as $g\mapsto g^{-1}$ on elements
of $G$, but acts as $\iota$ on $K$.

If $(R,\sigma)$ is any ring with involution, its isometry group $\Iso(R,\sigma)$
is the set of $x\in R$ such that $x\sigma(x)=1$ (it is a subgroup of $R^{\times}$).
In particular, $G$ is a subgroup of $\Iso(K[G],S_\iota)$. We can improve on this 
observation, with an "involutory" version of the fact that $G\mapsto K[G]$
is left adjoint to $R\mapsto R^\times$:

\begin{prop}\label{prop_adj_inv_alg}
  Let $\mathbf{AlgInv}(K,\iota)$ be the category of $K$-algebras endowed with
  an involution acting by $\iota$ on $K$. The functor $G\mapsto (K[G],S_\iota)$ from the 
  category of groups to $\mathbf{AlgInv}(K,\iota)$ is left adjoint to 
  $(R,\sigma)\mapsto \Iso(R,\sigma)$.
\end{prop}

\begin{proof}
  Let $(R,\sigma)$ be in $\mathbf{AlgInv}(K,\iota)$ and let $G$ be a group.
  Let $f:G\to \Iso(R,\sigma)$. Then since $\Iso(R,\sigma)$ is a subgroup of
  $R^\times$, $f$ extends uniquely to a $K$-algebra morphism $K[G]\to R$.
  It is now clear by definition of $S_\iota$ that $f$ is a morphism from
  $(K[G],S_\iota)$ to $(R,\sigma)$.
\end{proof}

If we return to the Goldman morphism for Azumaya algebras:

\begin{prop}\label{prop_sym_inv}
  Let $(A,\sigma)$ be an Azumaya algebra with involution over $(K,\iota)$. Then for 
  any $d\in \N$ the canonical $K$-algebra morphism $K[\mathfrak{S}_d]\to A^{\otimes d}$ 
  is a morphism of involutive algebras 
  \[  (K[\mathfrak{S}_d],S_\iota)\to (A^{\otimes d},\sigma^{\otimes d}). \]

  Equivalently, the canonical group morphism $\mathfrak{S}_d\to A^{\otimes d}$
  actually takes values in the isometry group $\Iso(A^{\otimes d},\sigma^{\otimes d})$.
\end{prop}

\begin{proof}
  The equivalence of the two formulations is clear given Proposition \ref{prop_adj_inv_alg}. 
  We can then reduce to the case of $d=2$ and a
  transposition, which means we have to prove that the Goldman element is symmetric
  for $\sigma^2$. By definition of $g_A$, this amounts to the fact that if 
  $g_A = \sum_i a_i\otimes b_i$ with $a_i,b_i\in A$, then for any $x\in A$,
  $\sum_i \sigma(a_i)x\sigma(b_i) = \Trd_A(x)$. But that element is 
  $\sigma(\sum_i b_i\sigma(x)a_i)$, which is $\Trd_A(\sigma(x))=\Trd_A(x)$ because 
  $\sum_i b_i\otimes a_i = g_A(\sum_i a_i\otimes b_i)g_A = g_A$.
\end{proof}

\begin{coro}\label{cor_sd_sym}
  Let $(A,\sigma)$ be an Azumaya algebra with involution over $(K,\iota)$, and let $(V,h)$ 
  be an $\eps$-hermitian module over $(A,\sigma)$. Then for any $d\in \N$,
  any $x,y\in V^{\otimes d}$, and any $\theta\in K[\mathfrak{S}_d]$ we have
  \[ h^{\otimes d}(\theta\cdot x, y) = h^{\otimes d}(x, S_\iota(\theta)\cdot y).  \]
  In addition, $S_\iota(s_d)=s_d$, so we get $h^{\otimes d}(x,s_d y) 
  = h^{\otimes d}(s_d x,y)$.
\end{coro}

\begin{proof}
  Let $B=\End_A(V)$ and $\tau = \sigma_h$, and write $\theta_B\in B^{\otimes d}$
  for the image of $\theta$ by the canonical morphism. Then proposition
  \ref{prop_sym_inv} shows that $\tau^{\otimes d}(\theta_B)$ is the image
  of $S_\iota(\theta)$ in $B^{\otimes d}$, which shows the first formula by definition
  of the adjoint involution. The fact that $S_\iota(s_d)=s_d$ is clear since $g\mapsto g^{-1}$
  is bijective on $\mathfrak{S}_d$ and preserves $(-1)^g$.
\end{proof}

This observation allows the following definition:

\begin{defi}\label{def_alt_h}
  Let $(A,\sigma)$ be an Azumaya algebra with involution over $(K,\iota)$,
  and let $(V,h)$ be an $\eps$-hermitian module over $(A,\sigma)$. We set:
  \[ \foncdef{\Alt^d(h)}{\Alt^d(V)\times \Alt^d(V)}{A^{\otimes d}}
    {(s_d x, s_d y)}{h^{\otimes d}(s_d x,y) = h^{\otimes d}(x,s_d y).} \]
\end{defi}

This is well-defined according to corollary \ref{cor_sd_sym}, since 
$h^{\otimes d}(s_d x,y)$ only depends on $s_d x$ and not the full $x$,
and conversely $h^{\otimes d}(x,s_d y)$ only depends on $s_d y$.

The definition can be rephrased as
\begin{equation}
  \Alt^d(x_1\# \dots \# x_d, y_1\# \dots \# y_d) 
    = h^{\otimes d}(x_1\# \dots \# x_d, y_1\otimes \dots \otimes y_d).
\end{equation}

\begin{prop}
  Let $(A,\sigma)$ be an Azumaya algebra with involution over $(K,\iota)$, 
  and let $(V,h)$ be an $\eps$-hermitian module over $(A,\sigma)$. The application $\Alt^d(h)$ 
  is an $\eps^d$-hermitian form over $(A^{\otimes d},\sigma^{\otimes d})$.
\end{prop}

\begin{proof}
  We have for all $x,y\in V^{\otimes d}$ and all $a,b\in A^{\otimes d}$:
  \begin{align*}
    \Alt^d(h)(s_dx\cdot a,s_dy\cdot b) &= h^{\otimes d}(xa,s_dyb) \\
                                       &= \sigma^{\otimes d}(a)h^{\otimes d}(x,s_dy)b \\
                                       &= \sigma^{\otimes d}(a)\Alt^d(h)(s_dx,s_dy)b
  \end{align*}
  and
  \begin{align*}
    \Alt^d(h)(s_dy,s_dx) &= h^{\otimes d}(y,s_dx) \\
                         &= \eps^d\sigma^{\otimes d}(h^{\otimes d}(s_dx,y)) \\
                         &= \eps^d\sigma^{\otimes d}(\Alt^d(h)(s_dx,s_dy)). \qedhere
  \end{align*}
\end{proof}

\begin{rem}
  Clearly this construction works in the following setting: if $(V,h)$ is an $\eps$-hermitian
  space over $(A,\sigma)$, with adjoint algebra with involution $(B,\tau)$, and
  $b\in \Sym(B^\times, \tau)$, then $(b\cdot x, b\cdot y)\mapsto h(x,b\cdot y)$ is
  well-defined and is an $\eps$-hermitian form on $bV$.
\end{rem}

\begin{ex}
  When $A=K$ and $h$ is a hermitian form on the vector space $V$, then
  if $x=u_1\otimes\cdots\otimes u_d$ and $y=v_1\otimes\cdots\otimes v_d$,
  we get
  \[ \Alt^d(h)(s_dx,s_dy)=\sum_{\pi\in \mathfrak{S}_d}(-1)^\pi \prod_ih(u_{\pi^{-1}(i)},v_i) = \det(h(u_i,v_j)) \]
  so when we identify $\Alt^d(V)$ and $\Lambda^d(V)$, $\Alt^d(h)$ does
  correspond to $\lambda^d(h)$.
\end{ex}

The following simple observation is extremely usful in applications:

\begin{prop}
  Let $(A,\sigma)$ be an Azumaya algebra with involution over $(K,\iota)$, and let $(V,h)$ 
  be an $\eps$-hermitian module over $(A,\sigma)$. For any $d\in \N$ and any
  $\lambda\in K^\times$, we have 
  \[ \Alt^d(\fdiag{\lambda}_\iota h) = \fdiag{\lambda^d}_\iota \Alt^d(h). \]
\end{prop}

\begin{proof}
  This follows the definition of from $\Alt^d(h)$ and the fact that 
  $(\fdiag{\lambda}_\iota h)^{\otimes d}\simeq \fdiag{\lambda^d}_\iota h^{\otimes d}$.
\end{proof}

Since $\Alt^d(V)\subset V^{\otimes d}$, we can compare $\Alt^d(h)$ and $h^{\otimes d}$
on $\Alt^d(V)$.

\begin{prop}\label{prop_restr_alt}
  Let $(A,\sigma)$ be an Azumaya algebra with involution over $(K,\iota)$, and let $(V,h)$ 
  be an $\eps$-hermitian module over $(A,\sigma)$. For any $d\in \N$,
  we can restrict the $\eps^d$-hermitian form $h^{\otimes d}$ to 
  $\Alt^d(V)\subseteq V^{\otimes d}$, and we get
  \[ h^{\otimes d}_{| \Alt^d(V)} = \fdiag{d!}\Alt^d(h). \]
\end{prop}

\begin{proof}
  Since $s_d$ is symmetric, we have $h^{\otimes d}(s_dx,s_dy)=h^{\otimes d}(x,(s_d)^2y)$.
  But it is easy to see that $s_d^2=(d!)s_d$, which concludes.
\end{proof}

Note that this means that we could have simply defined $\Alt^d(h)$ in terms
of the restriction of $h^{\otimes d}$ in characteristic $0$, but in arbitrary
characteristic this does not work.

\begin{ex}\label{ex_alt_0_1}
  This shows in any characteristic that $\Alt^0(h)=\fdiag{1}$ and
  $\Alt^1(h)=h$.
\end{ex}

We can then show the compatibility of this construction with the sum formula:

\begin{prop}\label{prop_add_alt}
  Let $(A,\sigma)$ be an Azumaya algebra with involution over $(K,\iota)$, and let $(U,h)$
  and $(V,h')$ be $\eps$-hermitian modules over $(A,\sigma)$. The module
  isomorphism in Proposition \ref{prop_add_mod_alt} induces an isometry 
  \[ \bigoplus_{p+q=d} \Alt^p(h)\otimes_K \Alt^q(h') \Isom \Alt^d(h\perp h'). \]
\end{prop}

\begin{proof}
  Let $u,u'\in U^{\otimes p}$ and $v,v'\in V^{\otimes q}$. Then
  \begin{align*}
    & \Alt^d(h\perp h')((s_pu) \# (s_qv), (s_pu) \# (s_qv)) \\
    =& \Alt^d(h\perp h')(s_d(u\otimes v), s_d(u'\otimes v')) \\
    =& (h\perp h')^{\otimes d}(s_d(u\otimes v),u'\otimes v') \\
    =& \sum_{\pi\in \mathfrak{S}_d}(-1)^\pi 
      (h\perp h')^{\otimes d}(\pi(u\otimes v),u'\otimes v'),
  \end{align*}
  where we used proposition \ref{prop_shuffle_alg} for the first equality.
  We want to show that
  $(h\perp h')^{\otimes d}(\pi(u\otimes v),u'\otimes v') = 0$
  if $\pi\not\in \mathfrak{S}_{p,q}$. But if $u = x_1\otimes\cdots\otimes x_p$,
  $u' = y_1\otimes\cdots\otimes y_p$, and $v= x_{p+1}\otimes\cdots\otimes x_d$,
  $v'= y_{p+1}\otimes\cdots\otimes y_d$, then using proposition \ref{prop_goldman_commut}:
  \begin{align*}
    & (h\perp h')^{\otimes d}(\pi(u\otimes v),u'\otimes v') \\
    &= (h\perp h')^{\otimes d}((x_{\pi^{-1}(1)}\otimes \cdots \otimes x_{\pi^{-1}(d)})\pi,
      (y_1\otimes \cdots \otimes y_d)) \\
    &= \pi^{-1} (h\perp h')(x_{\pi^{-1}(1)}, y_1)\otimes \cdots \otimes
      (h\perp h')(x_{\pi^{-1}(d)},y_d)
  \end{align*}
  which is indeed zero if $\pi\not\in \mathfrak{S}_{p,q}$ since at least
  one of the $(h\perp h')(x_{\pi^{-1}(i)}, y_i)$ will be zero. Hence:
  \begin{align*}
    & \Alt^d(h\perp h')((s_pu) \# (s_qv), (s_pu) \# (s_qv)) \\
    &= \sum_{\pi\in \mathfrak{S}_{p,q}}(-1)^\pi 
      (h\perp h')^{\otimes d}(\pi(u\otimes v),u'\otimes v') \\
    &= \sum_{\pi_1\in \mathfrak{S}_p}\sum_{\pi_2\in \mathfrak{S}_q}
      (-1)^{\pi_1\pi_2} (h\perp h')^{\otimes d}(\pi_1u\otimes \pi_2v),u'\otimes v') \\
    &= h(s_pu,u')\otimes h'(s_qv,v'). \qedhere
  \end{align*}
\end{proof}

\begin{rem}
  If $d\ppq \deg(A)$, then the hermitian form $\Alt^d(\fdiag{1}_\sigma)$ induces
  an adjoint involution $\sigma^{\wedge d}$ on $\lambda^d(A)$. This is essentially
  the same definition of $\sigma^{\wedge d}$ as in \cite{BOI} (and it is indeed the
  same involution). But defining it at the level of hermitian forms instead
  of involutions allows to study the interplay with the additive structure,
  and therefore the pre-$\lambda$-ring structure.

  In general, if $B=\End_A(V)$ and $d\ppq \rdim(V)$, then the adjoint involution
  of $\Alt^d(h)$, defined on $\lambda^d(B)$, is $\sigma_h^{\wedge d}$.
\end{rem}

\subsection{The pre-$\lambda$-(semi)ring structures}

We now show that the previous constructions do yield the expected
structure on our various semirings and rings. 

\begin{thm}\label{thm_lambda_swn}
  Let $(A,\sigma)$ be an Azumaya algebra with involution over $(K,\iota)$. 
  Then the operations
  \[ \Alt^d: SW^\eps(A^{\otimes n},\sigma^{\otimes n})\to 
      SW^{\eps^d}(A^{\otimes nd},\sigma^{\otimes nd})  \]
  for $d\in \N$ and $\eps\in U(K,\iota)$ turn $\SWN$ into a rigid $\Gamma_\N$-structured
  semiring, with augmentation $\hat{\rdim}$.

  Furthermore, $(A,\sigma)\mapsto \SWN$ is a functor from $\CBrhu$ to the category of 
  $\Gamma_\N$-structured semirings with lax morphisms.
\end{thm}

\begin{proof}
  The fact that the $\Alt^d$ define a graded pre-$\lambda$-semiring 
  structure follows simply from Proposition \ref{prop_add_alt} and
  Example \ref{ex_alt_0_1}.

  It is easy to see that $\hat{\rdim}$ is a graded semiring morphism (see also
  \cite[Prop 4.8]{Moi2}), so 
  we need to check that it is a $\lambda$-morphism, which is exactly 
  the content of Proposition \ref{prop_dim_alt}. To show that 
  $\SWN$ is rigid, note that by definition of the reduced
  dimension, if $h\in SW^\eps(A^{\otimes n},\sigma^{\otimes n})$ satisfies
  $\rdim(h)=0$, then $h=0$; moreover, the line elements in 
  $\SWN$ are exactly the $1$-dimensional hermitian forms in 
  $SW^\eps(A^{\otimes n},\sigma^{\otimes n})$ for the $n$ such that 
  $A^{\otimes n}$ is split, and clearly such elements are quasi-invertible:
  up to Morita equivalence, multiplication by such an element amounts to
  multplication by some $\fdiag{a}_\iota\in SW(K,\iota)$, which clearly induces an
  isomorphism of Witt semigroups.
  
  Only the functoriality is left to prove. Let $f:(B,\tau)\to (A,\sigma)$ be 
  a morphism in $\CBrhu$. We already know that the induced map $f_*$ on the 
  $\hat{SW}_\N^\bullet$ is a graded semiring isomorphism, which preserves the reduced 
  dimension; it remains to check that it preserves the $\lambda$-operations.
  So let $(V,h)$ be an $\eps$-hermitian module over 
  $(B^{\otimes n},\tau^{\otimes n})$, and let $d\in \N$. What we want to 
  prove is then 
  \begin{equation}\label{eq_fonct_alt}
    f^{\otimes nd}_*(\Alt^d(h)) = \Alt^d(f^{\otimes n}_*(h)).
  \end{equation}
  Replacing $f$ by $f^{\otimes n}$ if necessary, it is enough to treat 
  the case $n=1$.

  Let $(U,g)$ be the hermitian space over $(A,\sigma)$ which defines
  to $f$. Then the underlying module on the left-hand side of 
  (\ref{eq_fonct_alt}) is
  \[ (s_dV^{\otimes d})\otimes_{B^{\otimes d}}U^{\otimes d}, \]
  and on the right-hand side:
  \[ s_d(V\otimes_B U)^{\otimes d}. \]
  There is an obvious bimodule isomorphism between the two, given by
  \[ (v_1\# \dots \# v_d)\otimes (u_1\otimes\dots \otimes u_d) \mapsto
    (v_1\otimes u_1)\# \dots \# (v_d\otimes u_d), \]
  and if we look at the definitions of $g^{\otimes d}\circ \Alt^d(h)$
  and $\Alt^d(g\circ h)$, we see that we need to prove that for any
  $u_i,u'_i\in U$ and $v_i,v'_i\in V$,
  \[ g^{\otimes d}(u_1\otimes\dots\otimes u_d, 
  h^{\otimes d}(v_1\otimes\dots\otimes v_d,v'_1\#\dots\# v'_d)(u'_1\otimes\dots\otimes u'_d))\]
  is equal to
  \[ (g\circ h)^{\otimes d}((v_1\otimes u_1)\otimes\dots\otimes (v_d\otimes u_d),
  (v'_1\otimes u'_1)\#\dots\# (v'_d\otimes u'_d)). \]
  It is then a straightforward computation, using proposition 
  \ref{prop_goldman_commut}, that both expressions are equal to
  \[ \sum_{\pi\in\mathfrak{S}_d} (-1)^\pi
    \left[ \bigotimes_i g(u_i,h(v_i,v'_{\pi^{-1}(i)})u'_{\pi^{-1}(i)}) \right]\pi. \qedhere \]
\end{proof}

\begin{coro}\label{cor_lambda_Z}
  The graded semirings $\SWZ$, $\GWN$, and $\GWZ$ are naturally $\Gamma_\N$-
  and $\Gamma_\Z$-structured semirings, such that $\SWN$ (resp. $\SWZ$) is
  the positive structure on $\GWN$ (resp. $\GWZ$). They define functors
  from $\CBrhu$ to the category of structured semirings with lax morphisms, 
  such that the square
  \[ \begin{tikzcd}
    \SWN \arrow[hook]{r} \arrow[hook]{d}  & \SWZ \arrow[hook]{d} \\
    \GWN \arrow[hook]{r} & \GWZ 
  \end{tikzcd} \] 
  is a commutative diagram of structured semirings, natural over $\CBrhu$.
\end{coro}

\begin{proof}
  Recall that $\SWZ$ is simply a gluing of $\SWN$ and $\SWN[\tw{A}, \tw{\sigma}]$, 
  and the $\lambda$-operations therefore carry over to
  $\SWZ$, and Proposition \ref{prop_check_n} shows that this also defines a 
  graded pre-$\lambda$-ring structure on $\SWZ$. We can also use Proposition 
  \ref{prop_check_n} to show that the canonical map $\hat{\rdim}: \SWZ\to \N[\Z]$
  is a $\lambda$-morphism. 

  We need to show that $\SWZ$ is rigid. The fact that $\rdim(x)=0$ implies
  $x=0$ is just as clear as for $\SWN$. Let $x\in SW^\eps(A^{\otimes d},\sigma^{\otimes d})$ 
  be homogeneous of $\lambda$-dimension $1$. We may assume that $d\pgq 0$,
  otherwise we apply the same reasoning to $(\tw{A},\tw{\sigma})$.
  It is not as obvious as for $\SWN$ 
  that $x$ is graded-invertible, as multiplication by $x$ must also induce isomorphisms
  for the components of negative $\Z$-degree. But the fact that $\SWN$ is
  rigid shows that $\rdim(x)=1$ (Proposition \ref{prop_pos_dim}) and therefore 
  $A^{\otimes d}$ is split. Choosing an
  equivalence between $(A^{\otimes d},\sigma^{\otimes d})$ and $(K,\iota)$, we
  see that $x$ is the image by a ring morphism $\SWZ[K,\iota]\to \SWZ$ of
  some $1$-dimensional element $y\in \SWZ[K,\iota]$. Using Proposition \ref{prop_gw_mixte_split},
  we see that $y$ is actually invertible in $\SWZ[K,\iota]$, because $1$-dimensional
  elements in $SW^\bullet(K,\iota)$ are invertible. This shows that $x$ is invertible,
  and in particular graded-invertible.

  The fact that this $\Gamma_\Z$-structured semiring is functorial over $\CBrhu$
  follows directly from the corresponding statement for $\SWN$ and 
  $\SWN[\mbox{}^\iota A, \mbox{}^\iota \sigma]$.

  We get the structure on $\GWN$ and $\GWZ$ simply by applying Proposition \ref{prop_pos_groth},
  and functoriality follows immediately from the case of $\SWN$ and $\SWZ$, and the
  functoriality of Grothendieck rings. The statement about the commutative square
  is clear by construction.
\end{proof}

\begin{rem}
  Note that since $\GWN$ is a graded pre-$\lambda$-ring, it is
  in particular an ungraded pre-$\lambda$-ring, but $\SWN$ is
  \emph{not} a positive structure in this context, because the line elements
  are not invertible (only graded-invertible). On the other hand, $\SWZ$
  is an ungraded positive structure on $\GWZ$, as the grading is done
  over a group, so graded-invertible elements are invertible.
\end{rem}

\begin{coro}\label{cor_lambda_zd}
  If $\iota=\Id_K$, $\SW$ and $\GW$ are naturally $\Gamma$-structured
  semirings, where $\SW$ is rigid and is the positive structure on $\GW$.
  Furthermore, they define functors from $\CBrhu[K,\Id]$ to
  the category of $\Gamma$-structured semirings with lax morphisms, such
  that the square
  \[ \begin{tikzcd}
    \SWZ \arrow[twoheadrightarrow]{r} \arrow[hook]{d}  & \SW \arrow[hook]{d} \\
    \GWZ \arrow[twoheadrightarrow]{r} & \GW
  \end{tikzcd} \] 
  is a natural commutative diagram of structured semirings.
\end{coro}

\begin{proof}
  Since $\SWN\to \SW$ is a contraction, we can transfer the structure to $\SW$
  as in Proposition \ref{prop_contr_lambda}, which shows that it is rigid. We
  get the structure on $\GW$ either from the contraction $\GWN\to \GW$, or
  as the Grothendieck ring of $\SW$. It is straightfoward that these yield the
  same structure. The functoriality is easily established as in Corollary 
  \ref{cor_lambda_Z}, and the commutative square follows from the definition
  of the structure.
\end{proof}

If $(V,h)$ is an $\eps$-hermitian module over $(A,\sigma)$, its image by the
operation $\lambda^d$ in $\GW$ will be denoted
\[  (\Lambda^d(V),\lambda^d(h))\in SW^{\eps^d}(A^{\otimes r},\sigma^{\otimes r}),  \]
where $r\in \{0,1\}$ has the same parity as $d$. 

\begin{rem}\label{rem_lambda_class}
  Note that, unlike $(\Alt^d(V),\Alt^d(h))$,
  which is a well-defined hermitian module, only the isometry class of 
  $(\Lambda^d(V),\lambda^d(h))$ is well-defined, because it is constructed
  from a Morita equivalence. This distinction is a reason why it is often more 
  convenient to prove things in
  $\SWN$ first, where we can work with actual modules, and transfer the
  results to $\GW$ by Morita equivalence. But in Section \ref{sec_ralt}
  we construct an explicit representative $(\RAlt^{d}(V),\RAlt^d(h))$
  of the isometry class $(\Lambda^d(V),\lambda^d(h))$.
  
  Also note that the isomorphism class of $\Lambda^d(V)$ as a bimodule
  depends on $\sigma$ and not only $V$, but does not depend on $h$
  (see the constructions of $\RAlt^d(V)$ in Section \ref{sec_ralt}).
  Furthermore, if $d>\rdim_A(V)$, then $\Lambda^d(V)=0$.
\end{rem}

\begin{ex}\label{rem_lambda_split_grad}
  We know from Proposition \ref{prop_gw_mixte_split} that as a ring
  $\GW[K,\Id]\simeq GW^\bullet(K,\Id)[\Zd]$. We also know from Proposition 
  \ref{prop_lambda_graded_group_ring} that $GW^\bullet(K,\Id)[\Zd]$ has a canonical 
  $\Gamma$-graded pre-$\lambda$-ring structure,
  since by Example \ref{ex_gw_lambda} $GW^\bullet(K,\Id)$ is a $\mu_2(K)$-graded
  pre-$\lambda$-ring.
  Then actually $\tld{GW}(K,\Id)\simeq GW^\pm[\Zd]$ as $\Gamma$-graded
  pre-$\lambda$-rings.
\end{ex}

\begin{rem}\label{rem_lambda_fonct}
  If $f:(B,\tau)\to (A,\sigma)$ is a morphism in $\CBrh$,
  corresponding to the $\eps$-hermitian form $h$,
  then by definition $f_*(\fdiag{1}_\tau)=h$, and
  since $f_*$ is compatible with the $\lambda$-operations,
  we have $\lambda^d(h) = f_*(\lambda^d(\fdiag{1}_\tau))$. Thus
  to be able to compute the exterior powers of any $\eps$-hermitian
  form, we just need to be able to do the computation in the special
  case of diagonal forms $\fdiag{1}_\tau$ for any involution $\tau$.
\end{rem}

\section{Reduced alternating power and trace forms}\label{sec_ralt}

Given a hermitian space $(V,h)$ over $(A,\sigma)$, we have a reasonably
explicit description of $(\Alt^d(V), \Alt^d(h))$ in terms of the action of 
$\mathfrak{S}_d$ on $V^{\otimes d}$. When $\sigma$ is of the first kind, 
we are more interested on $\lambda^d(h)$, which is defined from $\Alt^d(h)$
through a Morita equivalence, and only exists as an isometry class, not
a concretely defined space (see Remark \ref{rem_lambda_class}).

In this section, we assume always assume that $\sigma$ is of the
first kind (so $\iota=\Id_K$), and we will define explicit spaces called "reduced 
alternating powers" which give representative of the isometry classes
$(\Lambda^d(V), \lambda^d(h))$, essentially by going through the Morita
equivalence that defines them. When $d$ is even this is a quadratic space, and
when $d$ is odd this is a hermitian space. Since the behaviour is a 
little different, those two cases are treated separately, though the basic 
idea is the same in both cases.

In the special case that $(V,h)$ is $(A,\fdiag{1}_\sigma)$ we give a simplified
construction, which is actually universal for even powers, 
and which shows in particular that if the characteristic of the base field is large
enough, $\lambda^{2d}(h)$ can be realized up to a scalar as a subform of an
involution trace form (see Remark \ref{rem_ralt_even_1}).

\subsection{Reduced alternating powers of even degree}

\subsubsection*{Reduced tensor powers of modules}

Let $(A,\sigma)$ be an Azumaya algebra with involution of the first kind over $K$.
Let $V$ be a right $A$-module and let $d\in \N$. We define 
\begin{equation}
  V^{[2d,\sigma]} = V^{\otimes d}\otimes_{A^{\otimes d}} \mbox{}^{\sigma}V^{\otimes d}
\end{equation}
where $\mbox{}^{\sigma}V$ is the left $A$-module which
is $V$ as a vector space, with the action $a\cdot v=v\cdot \sigma(a)$.
Note that $(\mbox{}^{\sigma}V)^{\otimes d}$ and $\mbox{}^{\sigma^{\otimes d}}(V^{\otimes d})$
are identical as $A^{\otimes d}$-modules, so our notation can be taken
to mean either one.

We call $V^{[2d,\sigma]}$ the $2d$th \emph{reduced tensor power} of $V$, which depends on 
$\sigma$. When $\sigma$ is understood from the context, we just write $V^{[2d]}$. 

If $x,y\in V^{\otimes d}$, we write $x\odot y$ for the element in $V^{[2d]}$
defined by $x\otimes y\in V^{\otimes d}\otimes_{A^{\otimes d}} \mbox{}^{\sigma}V^{\otimes d}$.
This is meant to distinguish this element from $x\otimes y\in V^{\otimes 2d}$.
By construction, as a $K$-vector space $V^{[2d]}$ is a quotient of $V^{\otimes 2d}$,
with the quotient map given by $x\otimes y\mapsto x\odot y$.

\begin{lem}\label{lem_isom_red_tens_pow}
  Let $(A,\sigma)$ be an Azumaya algebra with involution over $(K,\Id)$,
  and let $V$ be a right $A$-module.
  There is a natural isomorphism of $K$-vector spaces between $V^{[2d, \sigma]}$
  and $V^{\otimes 2d}\otimes_{A^{\otimes 2d}} |A^{\otimes d}|_{\sigma^{\otimes d}}$,
  given in both directions by $x\odot y\mapsto (x\otimes y)\otimes 1$ and
  $(x\otimes y)\otimes a \mapsto (xa)\odot y = x\odot (y\sigma^{\otimes d}(a))$.
\end{lem}

\begin{proof}
  It is a straightforward check that those maps are well-defined and mutually
  inverse. For instance, if $x,y\in V^{\otimes d}$ and $a\in A^{\otimes d}$,
  $(xa)\odot y$ and $x\odot (y\sigma^{\otimes d}(a))$ are respectively sent to 
  \[ (xa\otimes y)\otimes 1 = ((x\otimes y)\cdot (a\otimes 1))\otimes 1 = (x\otimes y)\otimes a \]
  and
  \[ (x\otimes y\sigma^{\otimes d}(a))\otimes 1 = 
  ((x\otimes y)\cdot (1\otimes \sigma^{\otimes d}(a)))\otimes 1 = (x\otimes y)\otimes a \]
  which are indeed equal, remembering the module structure of 
  $|A^{\otimes d}|_{\sigma^{\otimes d}}$.
\end{proof}

\begin{rem}
  We chose to define $V^{[2d,\sigma]}$ in such a way that by definition 
  $V^{[2d,\sigma]}= (V^{\otimes d})^{[2,\sigma^{\otimes d}]}$. We could also
  have chosen $V^{[2d,\sigma]}= (V^{[2,\sigma^{\otimes d}]})^{\otimes d}$,
  in which case in Lemma \ref{lem_isom_red_tens_pow} we would have replaced
  $|A^{\otimes d}|_{\sigma^{\otimes d}}$ by $|A|_{\sigma}^{\otimes d}$.
  Since those two $A^{\otimes 2d}$-modules are isomorphic, this does not
  change anything, and we found our convention more convenient overall.
\end{rem}

\subsubsection*{Reduced tensor powers of hermitian forms}

Now let us assume that $V$ carries an $\eps$-hermitian form $h$ over $(A,\sigma)$.
We define the symmetric bilinear form $h^{[2d]}$ on $V^{[2d]}$ by 
\begin{equation}
  h^{[2d]}(x\odot y, x'\odot y') = \Trd_{A^{\otimes d}}(\sigma^{\otimes d}(h^{\otimes d}(x,x'))h^{\otimes d}(y,y'))
\end{equation}
with $x,x',y,y'\in V^{\otimes d}$.

\begin{prop}\label{prop_red_power_herm}
  Let $(A,\sigma)$ be an Azumaya algebra with involution over $(K,\Id)$,
  and let $(V,h)$ be an $\eps$-hermitian module over $(A,\sigma)$.
  Under the vector space isomorphism given in Lemma \ref{lem_isom_red_tens_pow}, the
  bilinear space $(V^{[2d]},h^{[2d]})$ is identified with the composition
  \[ (|A^{\otimes d}|_{\sigma^{\otimes d}},T_{\sigma^{\otimes d}})\circ (V^{\otimes 2d},h^{\otimes 2d}). \]   

  In particular, the isometry class of $(V^{[2d]},h^{[2d]})$, as an element of $SW(K)$, 
  is the $2d$th power in $\SW$ of the isometry class of $(V,h)$ in $SW^{\eps}(A,\sigma)$.
\end{prop}

\begin{proof}
  By definition, the transfer of the composition $T_{\sigma^{\otimes d}}\circ h^{\otimes 2d}$
  to $V^{[2d]}$ is given by 
  \begin{align*}
    (x\odot y, x'\odot y') &\mapsto T_{\sigma^{\otimes d}}(1, h^{\otimes 2d}(x\otimes y,x'\otimes y')\cdot 1) \\
       &= T_{\sigma^{\otimes d}}(1, h^{\otimes d}(x,x')\sigma^{\otimes d}(h^{\otimes d}(y,y'))) \\
       &= \Trd_{A^{\otimes d}}(\sigma^{\otimes d}(h^{\otimes d}(x,x'))h^{\otimes d}(y,y')).
  \end{align*}
  By construction of $\SW$, the $2d$th power of the isometry class of $(V,h)$ is precisely
  the isometry class of $T_{\sigma^{\otimes d}}\circ h^{\otimes 2d}$.
\end{proof}

\subsubsection*{The canonical action of the symmetric group}

We see from Lemma \ref{lem_isom_red_tens_pow} that $V^{[2d]}$ is naturally a quotient
of $V^{\otimes 2d}$ as a vector space, and that if $B=\End_A(V)$ it can be seen as 
a quotient as left $B^{\otimes 2d}$-modules. In particular, it has a canonical
structure of $K[\mathfrak{S}_{2d}]$-module, and we still refer to this
as the Goldman action of $K[\mathfrak{S}_{2d}]$ (or $\mathfrak{S}_{2d}$)
on $V^{[2d]}$.
We wish to describe the restriction of this action to two specific subgroups of $\mathfrak{S}_{2d}$.

First the group $\mathfrak{S}_d\times \mathfrak{S}_d$ can be embedded in $\mathfrak{S}_{2d}$ by
identifying it with the Young subgroup $\mathfrak{S}_{d,d}$.

\begin{lem}\label{lem_action_young}
  Let $(A,\sigma)$ be an Azumaya algebra with involution over $(K,\Id)$, and
  let $V$ be a right $A$-module.
  Given $g,h\in \mathfrak{S}_d$ and $x,y\in V^{\otimes d}$, the action of 
  $\mathfrak{S}_d\times \mathfrak{S}_d$ on $V^{[2d]}$ is given by 
  $(g,h)\cdot (x\odot y) = (gx)\odot (hy)$ using the left Goldman action of 
  $\mathfrak{S}_d$ on $V^{\otimes d}$.
\end{lem}

\begin{proof}
  This is straightforward since the $B^{\otimes 2d}$-module structure of 
  $V^{[2d]}$ is the quotient of the $B^{\otimes 2d}$-module structure of
  $V^{\otimes 2d}$, so for any $a,b\in B^{\otimes d}$ and $x,y\in V^{\otimes d}$,
  we get $(a\otimes b)\cdot (x\odot y) = (ax)\odot (by)$. Since the action of
  $\mathfrak{S}_{2d}$ is defined through $B^{\otimes 2d}$, the lemma follows.
\end{proof}

Second, we can embed $(\Zd)^d$ in $\mathfrak{S}_{2d}$ by sending 
$(x_i)_{1\ppq i\ppq d}\in (\Zd)^d$ to the permutation which exchanges $i$ and
$i+d$ for all $1\ppq i\ppq d$ with $x_i=1$, and leaves the other elements unchanged.
In other words, the $i$th generator $(0,\dots,1,\dots,0)$ is sent to the transposition
$(i, i+d)$. 

For any $g\in \mathfrak{S}_{2d}$, we define its $\sigma$-signature $\eps(\sigma)^g$ 
as $\eps(\sigma)$ if $g$ is an odd permutation, and $1$ if $g$ is even. So when $\sigma$
is orthogonal $\eps(\sigma)^g$ is always $1$, and when $\sigma$ is symplectic this is
the ordinary signature of $g$.

Note that the permutation action of $\mathfrak{S}_{2d}$ on $V^{\otimes 2d}$, unlike
the Goldman action, does not factor to $V^{[2d]}$. For instance, if $d=2$, the
action of the transposition $(1, 2)$ does not give a well-defined map 
$(x_1\otimes x_2)\odot (x_3\otimes x_4)\mapsto (x_2\otimes x_1)\odot (x_3\otimes x_4)$.
On the other hand, when we restrict to $(\Zd)^d$:

\begin{lem}\label{lem_action_zd}
  Let $(A,\sigma)$ be an Azumaya algebra with involution over $(K,\Id)$, and
  let $V$ be a right $A$-module.
  Consider the permutation action of $\mathfrak{S}_{2d}$ on $V^{\otimes 2d}$.
  The restriction of this action to the subgroup $(\Zd)^d\subset \mathfrak{S}_{2d}$
  factors through the natural quotient map $V^{\otimes 2d}\to V^{[2d]}$, and the
  resulting action on $V^{[2d]}$ coincides with the Goldman action of this subgroup, 
  up to multiplication by the $\sigma$-signature.
\end{lem}

\begin{proof}
  We can easily reduce to the case where $d=1$, and prove that the action of $g_B\in B^{\otimes 2}$
  on $V\otimes_A V$ is given by $g_B\cdot (x\odot y)=\eps(\sigma)y\odot x$.

  We know from Proposition \ref{prop_goldman_commut} that $g_B(x\otimes y)g_A = y\otimes x$.
  The image of $(x\otimes y)g_A\in V^{\otimes 2}$ in $V^{[2]}$ is $xb\odot y$, where
  $b = \mu((\Id\otimes \sigma)(g_A))$, writing $\mu:A^{\otimes 2}\to A$ for the multiplication
  map. We just need to see that $b=\eps(\sigma)$.

  This can be shown by reducing to the split case, or it follows from \cite[Exercise I.12]{BOI},
  since $b$ is by construction equal to $g_A\cdot 1$, using the twisted action of 
  $A^{\otimes 2}$ on $|A|_\sigma$.
\end{proof}

Since $\mathfrak{S}_d\times \mathfrak{S}_d$ and $(\Zd)^d$ generate $\mathfrak{S}_{2d}$,
these lemmas fully characterize the action of $\mathfrak{S}_{2d}$.

\subsubsection*{Reduced alternating powers of a module}

Since the reduced tensor power $V^{[2d]}$ corresponds to the $A^{\otimes 2d}$-module
$V^{\otimes 2d}$ through the Morita equivalence given by $|A^{\otimes d}|_{\sigma^{\otimes d}}$,
we logically define the reduced alternating power as the corresponding subspace of $V^{[2d]}$.

\begin{defi}\label{def_ralt_mod_even}
  If $(A,\sigma)$ is an Azumaya algebra with involution of the first kind over $K$ 
  and $V$ is a right $A$-module, we define its $2d$th reduced alternating power 
  \[ \RAlt^{2d,\sigma}(V) = s_{2d}\cdot V^{[2d,\sigma]} \subset V^{[2d,\sigma]} \] 
  using the Goldman action of $\mathfrak{S}_{2d}$ on $V^{[2d]}$.
\end{defi}

Similarly to the reduced tensor powers, we usually drop the $\sigma$ from the notation
and just write $\RAlt^{2d}(V)$. By definition, $\RAlt^{2d}(V)$ is the image
through the canonical quotient map $V^{\otimes 2d}\to V^{[2d]}$ of the subspace 
$\Alt^{2d}(V)\subset V^{\otimes 2d}$.
\\

For any $1\ppq i\ppq d$, let $\tau_i: V^{[2d]}\to V^{[2d]}$ be the linear automorphism
which acts on $x\odot y$ by exchanging the $i$th tensor factors of $x$ and $y$.
We say that $x\in V^{[2d]}$ is an anti-mirror element if for any $i\in \{1,\dots,d\}$,
$\tau_i(x)=-\eps(\sigma)x$, and we write $AM^{2d,\sigma}(V)$ for the subspace of anti-mirror
elements.

\begin{prop}\label{prop_mirror}
  Let $(A,\sigma)$ be an Azumaya algebra with involution of the first kind over $K$ 
  and $V$ a right $A$-module. The subspace $\RAlt^{2d,\sigma}(V)$ of $V^{[2d]}$
  is the intersection of $AM^{2d,\sigma}(V)$ with
  $\Alt^d(V)\otimes_{A^{\otimes d}} \mbox{}^\sigma V^{\otimes d}$.
\end{prop}

\begin{proof}
  Consider the subgroups $\mathfrak{S}_d\times \mathfrak{S}_d$ and
  $(\Zd)^d$ in $\mathfrak{S}_{2d}$, as above. The group $\mathfrak{S}_{2d}$ 
  is generated by $\mathfrak{S}_d\times \{1\}$ and $(\Zd)^d$, so according
  to Lemma \ref{lem_ker_im_sd}, $s_{2d}V^{[2d]}$ is the intersection of 
  the $\ker(1-(-1)^g g)$ for $g\in \mathfrak{S}_d\times \{1\}$ and 
  $g\in (\Zd)^d$.

  Using Lemma \ref{lem_action_zd}, the intersection of the $\ker(1-(-1)^g g)$
  with $g\in (\Zd)^d$ is exactly $AM^{2d,\sigma}(V)$. And using Lemma
  \ref{lem_action_young}, the action of $\mathfrak{S}_d\times \{1\}$ is
  simply the Goldman action of $\mathfrak{S}_d$ on the left factor
  of $V^{[2d]}=V^{\otimes d}\otimes_{A^{\otimes d}} 
  \mbox{}^{\sigma} V^{\otimes d}$, so using again Lemma \ref{lem_ker_im_sd},
  the intersection of the $\ker(1-(-1)^g g)$ for $g\in \mathfrak{S}_d\times \{1\}$
  is $s_dV^{\otimes d}\otimes_{A^{\otimes d}} \mbox{}^{\sigma} V^{\otimes d}$,
  which gives the first equality.
\end{proof}

\subsubsection*{Reduced alternating powers of a hermitian form}

If we assume again that $V$ carries an $\eps$-hermitian form over $(A,\sigma)$,
we define the $2d$th reduced alternating power of $h$ as the symmetric
bilinear form defined on $\RAlt^{2d}(V)$ by
\begin{equation}\label{eq_def_ralt_h}
  \RAlt^{2d}(h)(s_{2d}x,s_{2d}y) = h^{[2d]}(s_{2d}x,y) = h^{[2d]}(x,s_{2d}y).
\end{equation}

\begin{prop}\label{prop_ralt_h}
  Let $(A,\sigma)$ be an Azumaya algebra with involution of the first kind over $K$,
  and let $(V,h)$ be an $\eps$-hermitian module over $(A,\sigma)$.
  Under the restriction of the vector space isomorphism given in Lemma \ref{lem_isom_red_tens_pow}, 
  the bilinear space $(\RAlt^{2d}(V),\RAlt^{2d}(h))$ is identified with the composition
  \[ (|A|_{\sigma^{\otimes d}},T_{\sigma^{\otimes d}})\circ (\Alt^{2d}(V),\Alt^{2d}(h)). \]   

  In particular, $(\Lambda^{2d}(V),\lambda^{2d}(h))$ is the isometry class of 
  $(\RAlt^{2d}(V),\RAlt^{2d}(h))$.
\end{prop}

\begin{proof}
  We know from Proposition \ref{prop_red_power_herm} that $(V^{[2d]},h^{[2d]})$ is identified 
  with the composition
  \[ (|A^{\otimes d}|_{\sigma^{\otimes d}},T_{\sigma^{\otimes d}})\circ (V^{\otimes 2d},h^{\otimes 2d}). \]
  Let $b$ be the bilinear form on $\RAlt^{2d}(V)$ such that under this identification
  $(\RAlt^{2d}(V),b)$ corresponds to 
  \[ (|A^{\otimes d}|_{\sigma^{\otimes d}},T_{\sigma^{\otimes d}})\circ (\RAlt^{2d}(V),\Alt^{2d}(h)). \]
  Then considering that by definition 
  \[ \Alt^{2d}(h)(s_{2d}x,s_{2d}y)=h^{\otimes 2d}(s_{2d}x,y) = h^{\otimes 2d}(x,s_{2d}y) \] 
  we must have $b$ satisfing formula (\ref{eq_def_ralt_h}).
\end{proof}

\begin{coro}\label{cor_restr_ralt}
  The restriction of $h^{[2d]}$ to $\RAlt^{2d}(V)\subset V^{[2d]}$ is 
  $\fdiag{(2d)!}\RAlt^{2d}(h)$.
\end{coro}

\begin{proof}
  This is a direct consequence of Proposition \ref{prop_restr_alt}, since 
  $h^{[2d]}$ and $\RAlt^{2d}(h)$ are obtained from $h^{\otimes 2d}$ and 
  $\Alt^{2d}(h)$ through the same Morita equivalence.
\end{proof}

\subsubsection*{Reduced alternating powers of $\fdiag{1}_\sigma$}

The descriptions we gave can be somewhat simplified when $(V,h) = (A,\fdiag{1}_\sigma)$.

\begin{lem}\label{lem_ident_ad}
  Let $(A,\sigma)$ be an Azumaya algebra with involution of the first kind over $K$.
  The map $x\odot y \mapsto x\sigma(y)$ is an isomorphism of left $A^{\otimes 2d}$-modules
  from $A^{[2d,\sigma]}$ to $|A^{\otimes d}|_{\sigma^{\otimes d}}$, with inverse 
  $x\mapsto x\odot 1 = 1\odot \sigma^{\otimes d}(x)$.
\end{lem}

\begin{proof}
  This is just composing the isomorphism in Lemma \ref{lem_isom_red_tens_pow} with the 
  canonical isomorphism between 
  $A^{\otimes 2d}\otimes_{A^{\otimes 2d}} |A^{\otimes d}|_{\sigma^{\otimes d}}$
  and $|A^{\otimes d}|_{\sigma^{\otimes d}}$.
\end{proof}

So as a vector space $A^{[2d,\sigma]}$ can be identified with $A^{\otimes d}$,
and the action of $\mathfrak{S}_{2d}$ on $A^{\otimes d}$ that we use is the one
coming from the left $A^{\otimes 2d}$-module structure of $|A^{\otimes d}|_{\sigma^{\otimes d}}$
(the action twisted by $\sigma^{\otimes d}$, recall (\ref{eq_twisted_action})).

Let us write $\RAlt^{2d}(A,\sigma)$ for the subspace of $A^{\otimes d}$ corresponding
to $\RAlt^{2d,\sigma}(A)\subset A^{[2d,\sigma]}$, ie $\RAlt^{2d}(A,\sigma) = s_{2d}A^{\otimes d}$.

For any $1\ppq i\ppq d$, write $\sigma_i = 1\otimes\cdots\otimes \sigma\otimes\cdots\otimes 1$,
with the $\sigma$ at the $i$th spot, and define the subspace of totally $\sigma$-antisymmetric
elements in $A^{\otimes d}$ as 
\[ TA^{2d}(A,\sigma) = \ens{x\in A^{\otimes d}}{\forall i\in \{1,\dots,d\},\, \sigma_i(x) = -\eps(\sigma)x}.  \]
In particular, $TA^{2d}(A,\sigma)\subset \Sym^{(-\eps(\sigma))^d}(A^{\otimes d},\sigma^{\otimes d})$.

\begin{prop}
  Let $(A,\sigma)$ be an Azumaya algebra with involution of the first kind over $K$.
  Under the identification $A^{[2d,\sigma]}\simeq A^{\otimes d}$ of Lemma \ref{lem_ident_ad}, 
  the subspace $AM^{2d,\sigma}(A)$ of anti-mirror elements corresponds to the subspace
  $TA^{2d}(A,\sigma)$ of totally $\sigma$-antisymmetric elements, and
  \[ \RAlt^{2d}(A,\sigma) = TA^{2d}(A,\sigma)\cap \Alt^d(A) \subset A^{\otimes d}. \] 
\end{prop}

\begin{proof}
  This is a direct consequence of Proposition \ref{prop_mirror}, as it is 
  easy to see by definition that under the identification $x\odot y \mapsto x\sigma(y)$,
  $\tau_i$ corresponds to $\sigma_i$.
\end{proof}

Finally, we can identify the reduced tensor power and alternating power of $\fdiag{1}_\sigma$.

\begin{prop}\label{prop_ralt_1}
  Let $(A,\sigma)$ be an Azumaya algebra with involution of the first kind over $K$.
  Under the identification $A^{[2d,\sigma]}\simeq A^{\otimes d}$ of Lemma \ref{lem_ident_ad},
  the bilinear form $\fdiag{1}_\sigma^{[2d]}$ corresponds to $T_{\sigma^{\otimes d}}$, 
  and $\RAlt^{2d}(\fdiag{1}_\sigma)$ to 
  \[  (s_{2d}x,s_{2d}y) \mapsto  (-\eps(\sigma))^d \Trd_{A^{\otimes d}}((s_{2d}x)y)
  = (-\eps(\sigma))^d \Trd_{A^{\otimes d}}(x(s_{2d}y)).  \]  
\end{prop}

\begin{proof}
  The bilinear form corresponding to $\fdiag{1}_\sigma^{[2d]}$ sends $(x,y)$ to
  \begin{align*}
    \fdiag{1}_\sigma^{[2d]}(x\odot 1, y\odot 1) 
    &= \Trd_{A^{\otimes d}}(\sigma^{\otimes d}(\fdiag{1}_\sigma^{\otimes 2d}(x,y))\fdiag{1}_\sigma^{\otimes 2d}(1,1)) \\
    &= \Trd_{A^{\otimes d}}(x\sigma^{\otimes d}(y))
  \end{align*}
  which is $T_{\sigma^{\otimes d}}$. Then according to (\ref{eq_def_ralt_h})
  the form corresponding to $\RAlt^{2d}(\fdiag{1}_\sigma)$ sends $(s_{2d}x,s_{2d}y)$ to
  \begin{align*}
    T_{\sigma^{\otimes d}}(s_{2d}x, y) &= \Trd_{A^{\otimes d}}(\sigma^{\otimes d}(s_{2d}x)y) \\
    &= (-\eps(\sigma))^d \Trd_{A^{\otimes d}}((s_{2d}x)y)
  \end{align*}
  where we use that $\sigma^{\otimes d}(s_{2d}x) = (-\eps(\sigma))^d s_{2d}x$, since
  $s_{2d}x$ is in $TA^{2d}(A,\sigma)$.

  The other equality follows since the bilinear form is symmetric.
\end{proof}

\begin{rem}
  We verified that our identifications did yield the expected bilinear forms, but 
  in itself the fact that $\fdiag{1}_\sigma^{2d}$ is the isometry class of 
  $T_{\sigma^{\otimes d}}$ is truly by definition of $\SW$.
\end{rem}

\begin{coro}\label{cor_lambda_trace}
  Let $(A,\sigma)$ be an Azumaya algebra with involution of the first kind over $K$.
  The restriction of $T_{\sigma^{\otimes d}}$ to $\RAlt^{2d}(A,\sigma)\subset A^{\otimes d}$
  is isometric to $\fdiag{(2d)!}\lambda^{2d}(\fdiag{1}_\sigma)$. Therefore, if $char(K)\ppq 2d$
  then the restriction of $T_{\sigma^{\otimes d}}$ to $\RAlt^{2d}(A,\sigma)$ is totally isotropic,
  and if $char(K)>2d$ then $\lambda^{2d}(\fdiag{1}_\sigma)$ is isometric to $\fdiag{(2d)!}$ times
  the restriction of $T_{\sigma^{\otimes d}}$ to $\Alt^d(A)\cap TM^{2d}(A,\sigma)$.
\end{coro}

\begin{proof}
  This is a direct consequence of Corollary \ref{cor_restr_ralt}.
\end{proof}

\begin{coro}\label{cor_lambda_2}
  Let $(A,\sigma)$ be an Azumaya algebra with involution of the first kind over $K$.
  Then in $SW(K)$:
  \[ \lambda^2(\fdiag{1}_\sigma) = \fdiag{2}T_\sigma^{-\eps(\sigma)}. \]
\end{coro}

\begin{rem}\label{rem_ralt_even_1}
  In fact, as we observed in Remark \ref{rem_lambda_fonct}, for any $(V,h)$ over
  $(A,\sigma)$, if $B=\End_A(V)$ and $\tau=\sigma_h$ is the adjoint involution, then
  $\lambda^{2d}(h) = \lambda^{2d}(\fdiag{1}_\tau)$, so Proposition \ref{prop_ralt_1}
  is enough to compute any even $\lambda$-power. We can even be more explicit:
  $B^{\otimes d}\simeq \End_{A^{\otimes d}}(V^{\otimes d})$, and the standard identification
  $V^{\otimes d}\otimes_{A^{\otimes d}} \mbox{}^{\sigma^{\otimes d}}V^{\otimes d}\simeq B^{\otimes d}$
  given in \cite[§5.A]{BOI} gives an identification between $V^{[2d]}$ and $B^{\otimes d}$,
  which identifies $\RAlt^{2d}(h)$ and $\RAlt^{2d}(\fdiag{1}_\tau)$.

  In particular, for any $h$, we can always realize $\lambda^{2d}(h)$ as a scaled subform
  of some involution trace form, as long as the characteristic of the field is strictly
  superior to $2d$.
\end{rem}

\begin{rem}
  Our description of $\RAlt^2(\fdiag{1}_\sigma)$ in Proposition \ref{prop_ralt_1}
  yields the following alternative description of $\lambda^2(\fdiag{1}_\sigma)$.
  If $\sigma$ is orthogonal, $\RAlt^2(A,\sigma)$ is the space of alternating elements
  of $A$ in the sense of \cite[§2.A]{BOI}, and $\RAlt^2(\fdiag{1}_\sigma)$ is 
  \[ (x-\sigma(x),y-\sigma(y)) \mapsto \Trd_A(x(y-\sigma(y))). \]
  If $\sigma$ is symplectic, $\RAlt^2(A,\sigma)$ is the space of symmetrized elements
  as in \cite[§2.A]{BOI}, and $\RAlt^2(\fdiag{1}_\sigma)$ is 
  \[ (x+\sigma(x),y+\sigma(y)) \mapsto \Trd_A(x(y+\sigma(y))). \]
  Those are the forms described in \cite[Exercise 2.15]{BOI}, which also
  make sense in characteristic $2$.
\end{rem}

\subsection{Reduced alternating powers of odd degree}

We now do a similar construction for odd $\lambda$-powers, heavily
relying on the even case. We give less details as we are mainly interested
in even $\lambda$-powers in applications.

If $V$ is a right $A$-module and $d\in \N$, we define
\begin{equation}
  V^{[2d+1,\sigma]} = V^{[2d,\sigma]}\otimes_K V.
\end{equation}

As before, we usually drop the $\sigma$ from the notation.
If $B=\End_A(V)$, we know $V^{[2d]}$ is a left $B^{\otimes 2d}$-module,
so $V^{[2d+1]}$ is naturally a $B^{\otimes 2d+1}$-$A$-bimodule.
In particular, it is a left $K[\mathfrak{S}_{2d+1}]$-module,
and we may define 
\begin{equation}
  \RAlt^{2d+1,\sigma}(V) = s_{2d+1}V^{[2d+1,\sigma]}\subset V^{[2d+1,\sigma]}.
\end{equation}

When $V$ carries an $\eps$-hermitian form we may define an 
$\eps$-hermitian form $h^{[2d+1]}$ on $V^{[2d+1]}$ by 
\begin{equation}
  h^{[2d+1]} = h^{[2d]}\otimes h
\end{equation}
and an $\eps$-hermitian form $\RAlt^{2d+1}(h)$ on $\RAlt^{2d+1}(V)$ by
\begin{equation}
  \RAlt^{2d+1}(h)(s_{2d+1}x,s_{2d+1}y) = h^{[2d+1]}(s_{2d+1}x,y).
\end{equation}

\begin{prop}
  Let $(A,\sigma)$ be an Azumaya algebra with involution of the first kind
  over $K$, and let $(V,h)$ be an $\eps$-hermitian module over $(A,\sigma)$.
  Then the isometry class of $(V^{[2d+1]},h^{[2d+1]})$ is the $(2d+1)$th
  power of $(V,h)$ in $\SW$, and $(\Lambda^{2d+1}(V),\lambda^{2d+1}(h))$ is
  the isometry class of $(\RAlt^{2d+1}(V), \RAlt^{2d+1}(h))$.
\end{prop}

\begin{proof}
  By definition of the structure of $\SW$, $(V,h)^{2d+1}$ is the composition
  of $(V^{\otimes 2d+1},h^{\otimes 2d+1})$ and 
  $(|A^{\otimes d}|_{\sigma^{\otimes d}}, T_{\sigma^{\otimes d}})\otimes_K (|A|,\fdiag{1}_\sigma)$.
  This composition is the tensor power of 
  \[ (|A^{\otimes d}|_{\sigma^{\otimes d}}, T_{\sigma^{\otimes d}})\circ (V^{\otimes 2d},h^{\otimes 2d}) \]   
  which we know to be $(V^{[2d]},h^{[2d]})$ from Proposition \ref{prop_red_power_herm}
  and of 
  \[ (A,\fdiag{1}_\sigma)\circ (V,h) \]   
  which is canonically $(V,h)$. So $(V,h)^{2d+1}$ is the class of $(V^{[2d+1]},h^{[2d+1]})$.

  The statement regarding $\lambda^{2d+1}(h)$ follows, using the connexion between $h^{\otimes 2d+1}$
  and $\Alt^{2d+1}(h)$, exactly as in the proof of Proposition \ref{prop_ralt_h}.
\end{proof}

\begin{rem}
  Unlike the case of even $\lambda$-powers (see Remark \ref{rem_ralt_even_1}), we cannot compute 
  $\lambda^{2d+1}(h)$ simply as $\lambda^{2d+1}(\fdiag{1}_\tau)$ where $\tau$ is the 
  adjoint involution of $h$. Rather, $\lambda^{2d+1}(h)\in SW^\eps(A,\sigma)$ is obtained from 
  $\lambda^{2d+1}(\fdiag{1}_\tau)\in SW(B,\tau)$ using the Morita equivalence from
  $(B,\tau)$ to $(A,\sigma)$ given precisely by $(V,h)$.
\end{rem}

\section{The determinant of an involution}

The determinant being one of the most basic and useful invariants
of quadratic forms, it makes sense that one would like to extend it
to algebras with involutions and hermitian forms. 

Let $(A,\sigma)$ be an Azumaya algebra with involution over $(K,\iota)$,
and let $(V,h)$ be an $\eps$-hermitian module over $(A,\sigma)$,
of reduced dimension $n\in \N$.Then applying Proposition \ref{prop_def_det} to $\SWN$, 
we may define $\det(h) = \lambda^n(h) \in \ell(\SWN)$.
In particular, we define the determinant of $(A,\sigma)$ (or just of $\sigma$) as
\begin{equation}
  \det(A,\sigma) = \det(\sigma) = \det(\fdiag{1}_\sigma)\in \ell(\SWN).
\end{equation}

Now $\ell(\SWN)$ is a submonoid of $\SWN^\times$, with a natural morphism
$\partial: \ell(\SWN)\to \Gamma_\N$, which we compose with the natural
projection $\Gamma_\N\to \N$ to get $\partial': \ell(\SWN)\to \N$.
It is clear that $SW^\bullet(A^{\otimes d},\sigma^{\otimes d})$ contains 
elements of reduced dimension $1$ only when $A^{\otimes d}$ is split,
and in that case a choice of Morita equivalence yields a (choice-dependent)
identification between the set of those line elements with $\ell(SW^\bullet(K,\iota))$.
Let $e\in \N$ be the exponent of $A$. Then the image of $\partial': \ell(\SWN)\to \N$
is $e\N$, and if $m\in \N$ the fiber above $em$ is in non-canonical
bijection with $\ell(SW^\bullet(K,\iota))\simeq K^\times/N_{K/k}(K^\times)$.
Actually, we see that any choice of Morita equivalence between 
$(A^{\otimes d},\sigma^{\otimes d})$ and $(K,\iota)$ yields a 
monoid isomorphism between $(\partial')^{-1}(d\N)$ and $d\N\times K^\times/N_{K/k}(K^\times)$,
so in particular a choice of equivalence between 
$(A^{\otimes e},\sigma^{\otimes e})$ and $(K,\iota)$ gives
\[ \ell(\SWN) \approx e\N\times K^\times/N_{K/k}(K^\times). \]  

Then $\det(h)$ is in the fiber above $n\in e\N$, but this only
non-canonically identifies $\det(h)$ with a class in $K^\times/N_{K/k}(K^\times)$.
When $n$ is a multiple of $r=\deg(A)$ (so when $h$ is isometric to a diagonal
form, unless $(A,\sigma)=(K,\Id)$ and $h$ is anti-symmetric), 
we can do a little better. Indeed, there is a canonical equivalence between
$(A^{\otimes d},\sigma^{\otimes d})$ and $(K,\iota)$, given by $(\Alt^r(A),\Alt^r(\fdiag{1}_\sigma))$.
This defines a canonical isomorphism
\[ \ell(\SWN)\supset (\partial')^{-1}(r\N) \simeq r\N\times K^\times/N_{K/k}(K^\times) \]   
which sends $\det(\sigma)$ to $(r,1)$. In other words, we are saying that
any element of $(\partial')^{-1}(rm)$ for some $m\in \N$ (and in particular
$\det(h)$ if $n=rm$) has the form $\lambda\cdot \det(\sigma)^m$ for a unique
class $\lambda\in K^\times/N_{K/k}(K^\times)\simeq \ell(SW^\bullet(K,\iota))$. 
We can be more explicit when a diagonalization of $h$ is given.

\begin{lem}
  Let $A$ be an Azumaya algebra over $K$, of degree $n$. Then for any 
  $a\in A^\times$, we have $s_n a^{\otimes n} = \Nrd_A(a)s_n$.
\end{lem}

\begin{proof}
  The equality can be checked after scalar extension, so
  it is enough to prove this when $A$ is split. In that case $A\simeq \End_K(U)$
  with $\dim(U)=n$,
  $a$ corresponds to some endomorphism $f:U\to U$, and the formula amounts to
  $f(u_1)\wedge\cdots\wedge f(u_n) = \det(f)(u_1\wedge\cdots\wedge u_n)$ for
  all $u_1,\dots,u_n\in U$.
\end{proof}

\begin{prop}
  Let $(A,\sigma)$ be an Azumaya algebra with involution over $(K,\iota)$,
  and let $a_1,\dots,a_m\in \Sym^{\eps}(A^\times,\sigma)$ for some $\eps\in U(K,\iota)$.
  Then 
  \[ \det(\fdiag{a_1,\dots,a_m}_\sigma)= \fdiag{\prod_{i=1}^m\Nrd_A(a_i)}_\iota\det(\sigma)^m. \]
\end{prop}  

\begin{proof}
  Since $\det(h+h')=\det(h)\det(h')$, we can easily reduce to $m=1$, and show
  $\Alt^n(\fdiag{a}_\sigma) = \fdiag{\Nrd_A(a)}_\iota \Alt^n(\fdiag{1}_\sigma)$.
  Let $x,y\in A^{\otimes n}$. Then 
  \begin{align*}
    \Alt^n(\fdiag{a}_\sigma)(s_nx,s_ny) &= \fdiag{a}_\sigma^{\otimes n}(s_nx,y) \\
    &= \sigma^{\otimes n}(s_nx)a^{\otimes n}y \\
    &= \sigma^{\otimes n}(x)s_na^{\otimes n}y \\
    &= \Nrd_A(a)\sigma^{\otimes n}(x)s_ny \\
    &= \Nrd_A(a)\fdiag{1}_\sigma^{\otimes n}(x,s_ny) \\
    &= \Nrd_A(a)\Alt^n(\fdiag{1}_\sigma)(s_nx,s_ny). \qedhere
  \end{align*}
\end{proof}

Tu summarize, $\det(h)$ is canonically an element of $\ell(\SWN)$,
can be identified with an element of $K^\times/N_{K/k}(K^\times)$
only given a choice of Morita equivalence, but if $\rdim(h)=m\deg(A)$
we can relate $\det(h)$ and $\det(\sigma)^m$ by a class in $K^\times/N_{K/k}(K^\times)$.
\\

As usual, when $\iota=\Id_k$, it is much more comfortable to work in $\SW$,
and define $\det(h)$ and in particular $\det(\sigma)=\det(\fdiag{1}_\sigma)$
as an element of $\ell(\SW)$ (which is a group).

When $A$ is not split, we easily see that $\ell(\SW)\simeq K^\times/(K^\times)^2$
since there are no line elements in $SW^\eps(A,\sigma)$. When $A$ is split, 
the morphism $\ell(\SW)\xrightarrow{\partial} \Gamma\to \Zd$ induces a canonical 
exact sequence
\begin{equation}
  1\to K^\times/(K^\times)^2 \to \ell(\SW)\to \Zd \to 0
\end{equation}
which is split, but non-canonically so. Indeed, any choice of
Morita equivalence between $(A,\sigma)$ and $(K,\Id)$ defines an
isomorphism 
\[ \ell(\SW)\approx \ell(\SW[K,\Id])\simeq \Zd\times K^\times/(K^\times)^2. \]
Any two such choices of equivalences differ by the multiplication by some
$\fdiag{\lambda}$ with $\lambda\in K^\times$, which induces the automorphism
of $\ell(\SW[K,\Id])$ which corresponds to $([i],[a])\mapsto ([i],[\lambda^i a])$
with $[i]\in \Zd$ and $[a]\in K^\times/(K^\times)^2$ (in particular, it is
the identity on the "even component" of $\ell(\SW[K,\Id])$).

When $\rdim(h)$ is even, $\det(h)$ is then canonically identified with
a class in $K^\times/(K^\times)^2$. This applies in particular to $\det(\sigma)$
when $\deg(A)$ is even. 

When $\rdim(h)$ is odd, necessarily $A$ is split, and $\det(h)$ is in 
the odd component of $\ell(\SW)$ which is non-canonically identified with
$K^\times/(K^\times)^2$. This applies to $\det(\sigma)$ when $\deg(A)$
is odd; in that case, for any $h$ with odd reduced dimension, we can
write $\det(h)=\fdiag{\lambda}\det(\sigma)$ and the square class of
$\lambda$ is uniquely determined.
\\

Still when $\iota=\Id_k$, let us compare these observations with the
reference $\cite{BOI}$. In there $\det(\sigma)\in K^\times/(K^\times)^2$ 
is defined only when $\sigma$ is orthognal and $\deg(A)$ is even, 
and in that case it coincides with our definition. Indeed, we can
find a splitting extension $L/K$ of $A$ such that 
$K^\times/(K^\times)^2\to L^\times/(L^\times)^2$ is injective
(for instance we can take $L$ to be the function field of the Severi-Brauer
variety of $A$). Then we can choose a Morita equivalence $(V,b)$ from
$(A_L,\sigma_L)$ to $(L,\Id)$, and $\det(\sigma_L)=\det(b)$ both
for our definition and that of \cite{BOI}, and by injectivity
this shows that both $\det(\sigma)$ are equal. The key point of course
is that this is independent of the choice of $b$, because any other
choice is isometric to $\fdiag{\lambda}b$ for some $\lambda\in L^\times$,
and $\det(\fdiag{\lambda}b)=\det(b)$ because $\dim(V)=\deg(A)$ is even.
With our point of view, we can say that $\det(\sigma_L)$ is in the
even component of $\ell(\SW[A_L,\sigma_L])$, so its image in 
$\ell(\SW[L,\Id])$ does not depend on the choice of Morita equivalence.

When $\sigma$ is orthogonal but $\deg(A)$ is odd, \cite{BOI} does 
not define $\det(\sigma)$, because the previous trick does not work:
$\det(\fdiag{\lambda}b)=\fdiag{\lambda}\det(b)$ because $\dim(V)$
is odd, so this class does depend on the choice of $b$. In our language,
$\det(\sigma_L)$ is in the odd component of $\ell(\SW[A_L,\sigma_L])$,
so its image in $\ell(\SW[L,\Id])$ does depend on the choice of Morita equivalence.
Our definition still provides a meaning for $\det(\sigma)$ in this case,
but it is not a square class. It can be related to one by choosing an
equivalence between $(A,\sigma)$ and $(K,\Id)$, but the induced isomorphism
$\ell(\SW)\approx \Zd\times K^\times/(K^\times)^2$ does depend
on this choice (at least on the odd component). In general,
if $(V,h)$ defines a Morita equivalence between $(B,\tau)$ and 
$(A,\sigma)$, the induced isomorphism $\ell(\SW[B,\tau])\isom \ell(\SW)$
sends $\det(\tau)$ to $\det(\sigma)$, but the group itself varies.

For a symplectic involution, the situation is simpler: $\det(\sigma)$
is simply trivial (this can be seen using the splitting trick above for instance).
This is not suprising since there is no non-trivial cohomological invariant
of degree $1$ for symplectic involutions (basically because symplectic groups
are connected). Instead one may find in the literature definitions of 
a "determinant" of symplectic involutions as cohomological invariants
of degree 3, see \cite{BMT} and \cite{GPT}.

\bibliographystyle{plain}
\bibliography{lambda_hermitian}

\end{document}

%% file: canevas_anglais.tex
\usepackage[utf8]{inputenc}
\usepackage[english]{babel}
\usepackage[T1]{fontenc}
\usepackage{amsmath}
\usepackage{amsfonts}
\usepackage{amssymb}
\usepackage{amsthm}
\usepackage{stmaryrd}
\usepackage{tikz}
\usepackage{tikz-cd}
\usepackage{enumitem}
\usepackage{todonotes} 
\usepackage{graphicx}
\usepackage{hyperref}

\DeclareMathOperator{\Spec}{Spec}

\DeclareMathOperator{\Hom}{Hom}
\DeclareMathOperator{\End}{End}

\DeclareMathOperator{\Ima}{Im}

\DeclareMathOperator{\Trd}{Trd}
\DeclareMathOperator{\Nrd}{Nrd}
\DeclareMathOperator{\Br}{Br}

\DeclareMathOperator{\Iso}{Iso}

\DeclareMathOperator{\Id}{Id}

\DeclareMathOperator{\Alt}{Alt}

\DeclareMathOperator{\Sym}{Sym}
\DeclareMathOperator{\rdim}{rdim}

\newcommand{\isom}{\stackrel{\sim}{\rightarrow}}
\newcommand{\Isom}{\stackrel{\sim}{\longrightarrow}}

\newcommand{\Z}{\mathbb{Z}}
\newcommand{\N}{\mathbb{N}}

\newcommand{\To}{\longrightarrow}

\newcommand{\fdiag}[1]{\langle #1\rangle}
\newcommand{\ens}[2]{\{ #1\,|\, #2\}}

\newcommand{\tld}{\widetilde}
\newcommand{\eps}{\varepsilon}

\newcommand{\pgq}{\geqslant}
\newcommand{\ppq}{\leqslant}

\newcommand{\CBrh}[1][K]{\mathbf{Br}_h(#1)}
\newcommand{\CBrhu}[1][K,\iota]{\mathbf{Br}_h(#1)}
\newcommand{\Zd}{\Z/2\Z}

\renewcommand{\phi}{\varphi}
\renewcommand{\bar}{\overline}
\renewcommand{\hat}{\widehat}

\newcommand{\foncdef}[5]{\begin{array}{rrcl}
#1 : & #2 & \To & #3 \\
 & #4 & \longmapsto & #5
\end{array}}

\newtheorem{thm}{Theorem}[section]
\newtheorem{prop}[thm]{Proposition}
\newtheorem{coro}[thm]{Corollary}
\newtheorem{lem}[thm]{Lemma}
\newtheorem{defi}[thm]{Definition}

\theoremstyle{definition}
\newtheorem{rem}[thm]{Remark}
\newtheorem{ex}[thm]{Example}